\documentclass[12pt]{amsart}
\usepackage{amsmath,latexsym,amsfonts,amssymb,amsthm}
\usepackage{hyperref}
\usepackage{mathrsfs}
\usepackage{graphicx,color}
\usepackage{tikz-cd}
\usepackage{enumitem}
\usepackage{geometry}
\newtheorem{prop}{Proposition}
\newtheorem{thm}[prop]{Theorem}
\newtheorem{cor}[prop]{Corollary}

\newtheorem{lem}[prop]{Lemma}

\theoremstyle{definition}

\newtheorem{defn}[prop]{Definition}

\newtheorem{rem}[prop]{\it Remark}

\newtheorem*{ps}{Postscript Remark}

\numberwithin{equation}{section}

\newcommand{\bP}{\mathbb{P}}
\newcommand{\bC}{\mathbb{C}}
\newcommand{\bR}{\mathbb{R}}
\newcommand{\bA}{\mathbb{A}}
\newcommand{\bQ}{\mathbb{Q}}
\newcommand{\bZ}{\mathbb{Z}}
\newcommand{\bD}{\mathbb{D}}
\newcommand{\bE}{\mathbb{E}}
\newcommand{\bG}{\mathbb{G}}

\newcommand{\cX}{\mathcal{X}}
\newcommand{\cY}{\mathcal{Y}}

\newcommand{\cO}{\mathcal{O}}
\newcommand{\cL}{\mathcal{L}}
\newcommand{\cI}{\mathcal{I}}
\newcommand{\cM}{\mathcal{M}}
\newcommand{\cF}{\mathcal{F}}
\newcommand{\cJ}{\mathcal{J}}

\newcommand{\fa}{\mathfrak{a}}
\newcommand{\fb}{\mathfrak{b}}
\newcommand{\fc}{\mathfrak{c}}
\newcommand{\fm}{\mathfrak{m}}
\newcommand{\fp}{\mathfrak{p}}
\newcommand{\fq}{\mathfrak{q}}

\newcommand{\hX}{\hat{X}}
\newcommand{\hY}{\hat{Y}}

\newcommand{\Proj}{\mathbf{Proj}}
\newcommand{\Spec}{\mathbf{Spec}~}
\newcommand{\Supp}{\mathrm{Supp}}
\newcommand{\lct}{\mathrm{lct}}
\newcommand{\vol}{\mathrm{vol}}
\newcommand{\hvol}{\widehat{\vol}}
\newcommand{\mult}{\mathrm{mult}}
\newcommand{\const}{\mathrm{const}}
\newcommand{\rom}[1]{\lowercase\expandafter{\romannumeral #1\relax}}
\newcommand{\Val}{\mathrm{Val}}
\newcommand{\ord}{\mathrm{ord}}
\newcommand{\CM}{\mathrm{CM}}
\newcommand{\Ding}{\mathrm{Ding}}
\newcommand{\Sym}{\mathrm{Sym}}
\newcommand{\tr}{\mathrm{tr}}
\newcommand{\Ric}{\mathrm{Ric}}
\newcommand{\val}{\mathrm{val}}
\newcommand{\QM}{\mathrm{QM}}

\newcommand\numberthis{\addtocounter{equation}{1}\tag{\theequation}}

\begin{document}

\title[Volume of singular KE Fano varieties]{The volume of singular K\"ahler-Einstein
Fano varieties}
\author{Yuchen Liu}
\email{yuchenl@math.princeton.edu}
\address{Department of Mathematics, Princeton University,
Princeton, NJ, 08544-1000, USA.}
\date{\today}
%\classification{14J45, 14B05 (primary),  14L24, 32Q15 (secondary).}
%\keywords{Fano varieties, singularities, K-stability, K\"ahler-Einstein metrics.}
%\thanks{}
\begin{abstract}
 We show that the anti-canonical volume of an $n$-dimensional K\"ahler-Einstein 
 $\bQ$-Fano variety is bounded from above by certain invariants of the local singularities, namely $\lct^n\cdot\mult$ for ideals
 and the normalized volume function for real valuations. This refines a recent result by Fujita.
 As an application, we get sharp volume upper
 bounds for K\"ahler-Einstein Fano varieties with quotient singularities.
 Based on very recent results by Li and the author, we show that a Fano manifold is K-semistable
 if and only if a de Fernex-Ein-Musta\c t\u a type inequality
 holds on its affine cone. 
\end{abstract}

\maketitle

\section{Introduction}

An $n$-dimensional complex projective variety $X$ is said to be a {\it $\bQ$-Fano variety}
if $X$ has klt singularities and $-K_X$ is an 
ample $\bQ$-Cartier divisor. When $n\geq 2$, the (anti-canonical) volume 
$((-K_X)^n)$ of a $\bQ$-Fano variety $X$ can be arbitrarily large, such as
volumes of weighted projective spaces (see \cite{dol82}).
In the smooth case, Koll\'ar, Miyaoka, Mori \cite{kmm92} and
Campana \cite{cam92} showed that there exists a uniform volume upper bound $C_n$
for $n$-dimensional Fano manifolds. If $n\leq 3$, then $\bP^n$ has the 
largest volume among all $n$-dimensional Fano manifolds. This fails immediately
when $n\geq 4$ by examples of Batyrev \cite{bat81}.
However, if a $\bQ$-Fano variety $X$ admits
a K\"ahler-Einstein metric (in the sense of \cite{bbegz}),
then it is expected that $\bP^n$ has the largest volume
(see \cite{tia90,oss16} for $n=2$, and \cite{bb11,bb12} for smooth cases).
Recently, Fujita \cite{fuj15} showed that
if $X$ is a $n$-dimensional $\bQ$-Fano variety admitting a 
K\"ahler-Einstein metric, then the anti-canonical
volume $((-K_X)^n)$ is less than or equal to $(n+1)^n$ ($=$ the volume
of $\bP^n$).

In this paper, we refine the result \cite{fuj15} by looking at invariants
of the local singularities. Our first 
main result is the following:

\begin{thm}\label{mainthm}
Let $X$ be a K\"ahler-Einstein $\bQ$-Fano variety. 
Let $p\in X$ be a 
closed point. Let $Z$ be a closed subscheme of $X$ with $\Supp~Z=\{p\}$.
Then we have 
\begin{equation}\label{ineq1}
 ((-K_X)^n)\leq \left(1+\frac{1}{n}\right)^n\lct(X;I_Z)^n
 \mult_Z X,
\end{equation}
where $\lct(X;I_Z)$ is the log canonical threshold of the ideal
sheaf $I_Z$, and
$\mult_Z X$ is the Hilbert-Samuel multiplicity of $X$ along $Z$.
\end{thm}

Note that the invariant $\lct^n\cdot\mult$ has been studied by de Fernex, Ein and 
Musta\c t\u a in \cite{mus02,dfem03,dfem04,dfm15}. 

There are three major aspects of Theorem \ref{mainthm}:

\begin{itemize}
\item If we take $Z=p$ to be a smooth point on $X$, then $\lct(X;\fm_p)=n$
and $\mult_p X=1$. In particular, we recover Fujita's result \cite[Corollary 1.3]{fuj15}.

\item By definition, $\lct(X;I_Z)$ is the infimum of $\frac
{1+\ord_E(K_{Y/X})}{\ord_E(I_Z)}$ among all prime divisors $E$ on a resolution
$Y$ of $X$. Although $\lct(X;I_Z)$ is not easy to compute in general, we can always get a volume upper bound
by considering one divisor $E$, e.g. Theorem \ref{quotsing} and \ref{nonterm}.

\item On the other hand, if we fix a singularity $p$ on a K\"ahler-Einstein $\bQ$-Fano
variety $X$, we get a lower bound of $\lct^n\cdot\mult$ among all thickenings
of $p$. For example, if $X=\bP^n$ then the inequality \eqref{ineq1} recovers
the main result in \cite{dfem04} (see Theorem \ref{dfem}). More generally,
we get de Fernex-Ein-Musta\c{t}\u{a} type inequalities for all cone singularities
with K-semistable base using a logarithmic version of Theorem \ref{mainthm}
(see Theorem \ref{k-ss-dfem}).
\end{itemize}
\medskip

For a real valuation $v$ on $\bC(X)$ centered at a closed point $p$, the 
normalized volume function $\hvol$, introduced by C. Li in \cite{li15a}, is defined as
$\hvol(v):=A_X(v)^n\cdot\vol(v)$, where $A_X(v)$ is the log discrepancy of $v$ and $\vol(v)$ is the volume of $v$.
The following theorem gives an upper bound for the anti-canonical volume of $X$ in terms of the normalized volume of a real valuation. 

\begin{thm}\label{mainthm2}
Let $X$ be a K\"ahler-Einstein $\bQ$-Fano variety. 
Let $p\in X$ be a 
closed point. Let $v$ be a real valuation on $\bC(X)$ centered at $p$.
Then we have 
\begin{equation}\label{ineq2}
 ((-K_X)^n)\leq \left(1+\frac{1}{n}\right)^n\hvol(v).
\end{equation}
\end{thm}

Minimizing the normalized volume function $\hvol$ has been studied by Li, Xu, Blum 
and the author in \cite{li15a, li15b, ll16, lx16, blu16}. In \cite{ll16}, a logarithmic version of 
Theorem \ref{mainthm2} was developed to obtain the sharp lower bound
of the normalized volume function for cone singularities over K\"ahler-Einstein
$\bQ$-Fano varieties.

To compare Theorem \ref{mainthm} and \ref{mainthm2}, we know from \cite{mus02,li15a}
that the infimum of $\lct(X;I_Z)^n\cdot \mult_Z X$ among all thickenings $Z$ of $p$
is no bigger than the infimum of $\hvol(v)$ among all real valuations centered
at $p$. In fact, we show that these two infimums are the same
(see Theorem \ref{compare2}). 
In particular, this means that Theorem \ref{mainthm} and \ref{mainthm2} are
actually equivalent. 
\medskip

As an application of Theorem \ref{mainthm} and \ref{mainthm2},
we get sharp volume upper bounds for K\"ahler-Einstein $\bQ$-Fano varieties with quotient singularities.

\begin{thm}\label{quotsing}
 Let $X$ be a K\"ahler-Einstein $\bQ$-Fano variety of dimension $n$.
 Let $p\in X$ be a closed point.
 Suppose $(X,p)$ is a quotient singularity with local 
 analytic model $\bC^n/G$ where $G\subset GL(n,\bC)$ acts freely in
 codimension $1$. Then
 \[
  ((-K_X)^n)\leq \frac{(n+1)^n}{|G|},
 \]
with equality if and only if $|G\cap \mathbb{G}_m|=1$ and $X\cong\bP^n/G$,
where $\mathbb{G}_m\subset GL(n,\bC)$ is the subgroup consisting 
of non-zero scalar matrices.
\end{thm}

It is well-known that for algebraic surfaces, klt singularities are the
same as quotient singularities. In \cite{tia90}, Tian showed that
the anti-canonical volume of a K\"ahler-Einstein log Del Pezzo surface
is bounded from above by $48/|G|$ where $G\subset GL(2,\bC)$ is the 
orbifold group at a closed point. In \cite{oss16}, Odaka, Spotti and Sun
improved the volume upper bounds to $12/|G|$.

As a direct consequence of Theorem \ref{quotsing}, we get the sharp volume
upper bounds $9/|G|$ for K\"ahler-Einstein log Del Pezzo surfaces.

\begin{cor}\label{logdp}
Let $X$ be a K\"ahler-Einstein log Del Pezzo surface. 
Let $p\in X$ be a closed point with local analytic model $\bC^2/G$,
where $G\subset GL(2,\bC)$ acts freely in codimension $1$.
Then
\[
 ((-K_X)^2)\leq \frac{9}{|G|},
\]
with equality if and only if $|G\cap \mathbb{G}_m|=1$ and $X\cong\bP^2/G$.
\end{cor}

\begin{rem}
  The author was informed by Kento Fujita that he independently obtained cases of Corollary 
 \ref{logdp} when $(X,p)$ is Du Val.
\end{rem}

It was conjectured by Cheltsov and Kosta in \cite[Conjecture 1.18]{ck14}
that a Gorenstein Del Pezzo surface admits K\"ahler-Einstein metrics
if and only if the singularities are of certain types depending on
the anti-canonical volume. This conjecture was proved by Odaka, Spotti and Sun
in \cite[page 165]{oss16} where $\bQ$-Gorenstein smoothable
K\"ahler-Einstein log Del Pezzo surfaces are classified.
As a quick application of Corollary \ref{logdp}, we give an alternative
proof of the ``only if'' part of this conjecture.
We remark that partial results are known before \cite{oss16}, e.g.
\cite{dt92, jef97, won13} on the ``only if'' part and \cite{mm93, 
gk07, che08, shi10, ck14} on the ``if'' part.

\begin{cor}\label{ck}
Let $X$ be a K\"ahler-Einstein log Del Pezzo surface with at most Du Val
singularities. 
\begin{enumerate}
 \item If $((-K_X)^2)=1$, then $X$ has at most singularities of
 type $\bA_1$, $\bA_2$, $\bA_3$, $\bA_4$, $\bA_5$, $\bA_6$,
 $\bA_7$ or $\bD_4$.
 \item If $((-K_X)^2)=2$, then $X$ has at most singularities
of type $\bA_1$, $\bA_2$ or $\bA_3$.
 \item If $((-K_X)^2)=3$, then $X$ has at most singularities
of type $\bA_1$ or $\bA_2$.
 \item If $((-K_X)^2)=4$, then $X$ has at most singularities
of type $\bA_1$.
 \item If $((-K_X)^2)\geq 5$, then $X$ is smooth.
\end{enumerate}
\end{cor}

The comparison between Theorem \ref{mainthm} and \ref{mainthm2}
also allows us to rephrase the results from \cite{li15b,ll16} in 
terms of $\lct^n\cdot\mult$. The following theorem gives a necessary
and sufficient condition for a Fano manifold to be K-semistable.

\begin{thm}\label{k-ss-dfem}
 Let $V$ be a Fano manifold of dimension $n-1$.
 Assume $H=-rK_V$ is an ample Cartier divisor for 
 some $r\in\bQ_{>0}$. Let $X:=C(V,H)=\Spec\bigoplus_{k=0}^\infty
 H^0(V,mH)$ be the affine cone with cone vertex $o$.
 Then $V$ is K-semistable if and only if for any closed subscheme
 $Z$ of $X$ supported at $o$, the following inequality holds:
 \begin{equation}\label{eqdfemtype}
  \lct(X;I_Z)^n\cdot\mult_Z X\geq 
  \frac{1}{r}((-K_V)^{n-1}).
 \end{equation}
\end{thm}

In particular, from the ``only if'' part of Theorem \ref{k-ss-dfem} 
we get de Fernex-Ein-Musta\c t\u a type inequalities
\eqref{eqdfemtype} on cone singularities with K-semistable bases
(see \cite{dfem04} or Theorem \ref{dfem} for the smooth case).
Note that the ``if'' part is a straight-forward consequence
after \cite{li15b} using the relations between $\lct^n\cdot\mult$ and $\hvol$
from \cite{mus02,li15a}. 
\medskip

The proofs of Theorem \ref{mainthm} and \ref{mainthm2} rely
on the recent result \cite{fuj15} where Fujita 
settled the optimal volume upper bounds for Ding-semistable 
$\bQ$-Fano varieties. Recall that a $\bQ$-Fano variety $X$ is said to be \textit{Ding-semistable} if the Ding invariant $\Ding(\cX, \cL)$ is nonnegative for any normal test configuration $(\cX,\cL)/\bA^1$ of $(X,-rK_X)$. The results in \cite[Section 3]{ber12} show that, if a $\bQ$-Fano variety $X$ admits a K\"ahler-Einstein
metric, then $X$ is Ding-semistable (see also \cite[Theorem 3.2]{fuj15}). 
This allows us to prove Theorem \ref{mainthm} and \ref{mainthm2}
under the weaker assumption that $X$ is Ding-semistable (see Theorem \ref{main1} and \ref{main2}).
To prove the equality case of Theorem \ref{quotsing}, we blow
up a flag ideal associated to the $\hvol$-minimizing valuation to obtain 
a test configuration with zero $\CM$ weight. Then K-polystability is enough
to characterize the $\bQ$-Fano variety with maximal volume (see Lemma
\ref{kpoly} for details).
\medskip

The paper is organized as follows. In Section \ref{prelim}, we recall
the notions of Seshadri constants, Ding-semistability, filtrations 
and normalized volume of real valuations. As a generalization of 
\cite[Theorem 2.3]{fuj15}, we characterize the Seshadri constant of an arbitrary
thickening of a close point in terms of the volume function in 
Lemma \ref{volume}.
In Section \ref{pf}, we prove Theorem \ref{mainthm} and 
\ref{mainthm2}. The proof of Theorem \ref{mainthm} is a quick 
application of \cite[Theorem 1.2]{fuj15} and Lemma \ref{volume}.
To prove Theorem \ref{mainthm2}, we apply \cite[Theorem 4.9]{fuj15} to the 
filtration associated to any real valuation. In Section \ref{app}, 
we give some applications of Theorem \ref{mainthm} and \ref{mainthm2}. 
We recover the main result from \cite{dfem04}. We prove the inequality
part of Theorem \ref{quotsing}. We also get a volume upper bound for a
Ding-semistable $\bQ$-Fano variety with an isolated singularity that 
is not terminal (see Theorem \ref{nonterm}).
Section \ref{seccompare} is devoted to comparing the two invariants 
occurred in the volume upper bounds, i.e. $\lct^n\cdot\mult$ 
and $\hvol$. In Theorem \ref{compare2}, we show that these two 
invariants have the same infimums among ideals supported (resp.
valuations centered) at a klt point. This implies that Theorem 
\ref{mainthm} and \ref{mainthm2} are equivalent to each other. We prove
Theorem \ref{k-ss-dfem} using Theorem \ref{compare2} and results from 
\cite{li15b, ll16}. In Section \ref{minimizer}, we study the existence
of minimizers of $\lct^n\cdot\mult$ and $\hvol$. In Section 
\ref{maxvol}, we study the equality case of Theorem \ref{mainthm} 
and \ref{mainthm2}. Lemma \ref{kpoly} characterizes the 
equality case of Theorem \ref{mainthm2} when $X$ is 
K-polystable with extra conditions on the real valuation.
Applying this lemma, we prove the equality case of Theorem
\ref{quotsing} and Corollary \ref{ck}. We also prove that 
$\bP^n$ is the only Ding-semistable $\bQ$-Fano variety of 
volume at least $(n+1)^n$ (see Theorem \ref{Pn}). This theorem
improves the equality case of \cite[Theorem 1.1]{fuj15} where Fujita proved
for Ding-semistable Fano manifolds.

\subsection*{Notation}
Throughout this paper, we work over the complex numbers $\bC$.
Let $X$ be an $n$-dimensional normal variety. Let $\Delta$ be an effective
$\bQ$-divisor on $X$ such that $(X,\Delta)$ is a klt pair.
For a closed subscheme $Z$ of $X$, the \textit{log canonical threshold}
of its ideal sheaf $I_Z$ with respect to $(X,\Delta)$ is defined as
\[
 \lct(X,\Delta;I_Z):=\inf_E\frac{1+\ord_E(K_{Y}-f^*(K_X+\Delta))}{\ord_E(I_Z)},
\]
where the infimum is taken over all prime divisors $E$ on a log 
resolution $f:Y\to(X,\Delta)$. We also denote by $\lct(X;I_Z)=\lct(X,0;I_Z)$
when $\Delta=0$.
If $\dim Z=0$, the \textit{Hilbert-Samuel multiplicity}
of $X$ along $Z$ is defined as
\[
 \mult_Z(X):=\lim_{k\to\infty}\frac{\ell(\cO_X/I_Z^k)}{k^n/n!}.
\]
We will also use the notation $\lct(\fa)$ (and $\mult(\fa)$)
as abbreviation of $\lct(X;\fa)$ (and $\mult_{V(\fa)}X$)
for a coherent ideal sheaf $\fa$ (of dimension $0$) once the 
variety $X$ is specified.

Following \cite{bbegz, ber12},
on an $n$-dimensional $\bQ$-Fano variety $X$, a K\"ahler metric 
$\omega$ on the smooth locus $X_{reg}$ is said to be
a \textit{K\"ahler-Einstein metric on $X$} if it has constant Ricci curvature, 
i.e. $\Ric(\omega)=\omega$ on $X_{reg}$, and $\int_{X_{reg}}\omega^n=((-K_X)^n)$.
By \cite[Definition 3.5 and Proposition 3.8]{bbegz}, this is equivalent to saying that $\omega$ is
a closed positive $(1,1)$-current with full Monge-Amp\`ere mass
on $X$ satisfying $V^{-1}\omega^n=\mu_\omega$ (see \cite[Section 3]{bbegz}
for details).

Given a normal variety $X$ and a (not necessarily effective) $\bQ$-divisor $\Delta$, we say that
$(X,\Delta)$ is a \emph{sub pair} if $K_X+\Delta$ is $\bQ$-Cartier.
Let $f:Y\to (X,\Delta)$ be a log resolution, then we have the following equality
\[
 K_Y=f^*(K_X+\Delta)+\sum_{i}a_i E_i,
\]
where $E_i$ are distinct prime divisors on $Y$. We say that $(X,\Delta)$
is \emph{sub log canonical} if $a_i\geq -1$ for all $i$. If $D$ is a 
$\bQ$-Cartier $\bQ$-divisor on $X$, then the \emph{log canonical threshold}
of $D$ with respect to a sub pair $(X,\Delta)$ is defined as
\begin{equation}\label{sublct}
 \lct(X,\Delta;D):=\sup\{c\in\bR: (X,\Delta+cD)\textrm{ is sub log canonical}\}. 
\end{equation}

\subsection*{Acknowledgements}
 I would like to thank my advisor J\'anos
Koll\'ar for his constant support, encouragement
and numerous inspiring conversations. I wish to thank Chi Li and Chenyang Xu for 
fruitful discussions and encouragement. I also wish to thank
Charles Stibitz, Yury Ustinovskiy, Xiaowei Wang and
Ziquan Zhuang for many useful discussions, Lue Pan for providing
an argument for Lemma \ref{quotmult}, and Harold Blum, Kento Fujita and Song Sun
for helpful comments and suggestions through e-mails.
The author is partially supported by NSF grant DMS-0968337 and DMS-1362960.

\section{Preliminaries}\label{prelim}

\subsection{Seshadri constants}

\begin{defn}[\cite{laz04}]
 Let $X$ be a projective variety of dimension $n$.
 For a $\bQ$-Cartier $\bQ$-divisor $L$ on $X$,
 we define the \textit{volume} of $L$ on $X$ to be
 \[
  \vol_X(L):=\limsup_{k\to\infty,\, kL\textrm{ Cartier}}
  \frac{h^0(X,\cO_X(kL))}{k^n/n!}.
 \]
 We know that the limsup computing $\vol_X(L)$ is actually a limit.
 If $L\equiv L'$, then $\vol_X(L)=\vol_X(L')$. Moreover,
 we can extend it uniquely to a continuous function
 \[
  \vol_X:N^1(X)_{\bR}\to \bR_{\geq 0}.
 \]
\end{defn}

\begin{defn}
 Let $X$ be a projective variety, $L$ be an ample $\bQ$-Cartier divisor
 on $X$. Let $Z$ be a nonempty proper closed subscheme of $X$ with ideal
 sheaf $I_Z$. Denote by $\sigma:\hX\to X$ the blow up of $X$ along $Z$. 
 Let $F$ be the Cartier divisor on $\hX$ given by $\cO_{\hX}(-F)=\sigma^{-1}
 (I_Z)\cdot \cO_{\hX}$. The \textit{Seshadri constant} of $L$ along $Z$,
 denoted by $\epsilon_Z(L)$, is defined as
 \[
  \epsilon_Z(L):=\sup\{x\in\bR_{>0}\mid \sigma^*L-xF\textrm{ is ample}\}.
 \]
\end{defn}

\begin{lem}\label{volume}
Let $X$ be an $n$-dimensional normal projective variety with $n\geq 2$,
$L$ be an ample $\bQ$-divisor on $X$, $p\in X$ be a 
closed point. 
Let $Z$ be a closed subscheme of $X$ with $\Supp~Z=\{p\}$.
Let $\sigma:\hX\to X$ be the blow up along $Z$, 
and $F\subset\hX$ be the Cartier divisor defined by the 
equation $\cO_{\hX}(-F)=\sigma^{-1}I_Z\cdot\cO_{\hX}$.
\begin{enumerate}
\item For any $x\in\bR_{\geq 0}$, we have
\[
\vol_{\hX}(\sigma^*L-xF)\geq((\sigma^*L-xF)^n)=(L^n)-\mult_Z X \cdot x^n.
\]
\item Set $\Lambda_Z(L):=\{x\in\bR_{\geq 0}\mid \vol_{\hX}(\sigma^*L-xF)=((\sigma^*L-xF)^n)\}$. Then we have
\[
\epsilon_Z(L)=\max\{x\in\bR_{\geq 0}\mid y\in\Lambda_Z(L)\textrm{ for all }y\in[0,x]\}.
\]
\end{enumerate}
\end{lem}

\begin{proof} The proof goes along the same line as \cite[Theorem 2.3]{fuj15}.

(1) The equality $(\sigma^*L-xF)^n)=(L^n)-\mult_Z X \cdot x^n$ follows from
the fact that $\mult_Z(X)=(-1)^{n-1}(F^n)$ (see \cite{ram73}).

For the inequality part, we can assume that $x\in\bQ_{>0}$ since the function $\vol_{\hX}(\sigma^*L-xF)$ is continuous. Take any sufficiently
large $k\in\bZ_{>0}$ with $kx\in\bZ_{>0}$ and $kL$ Cartier.
Notice that we have the long exact sequence
\[
0\to H^0(\hX,\sigma^*(kL)-kxF)\to H^0(\hX, \sigma^*(kL))\to H^0(kxF,\sigma^*(kL)|_{kxF})\to\cdots\numberthis \label{longexseq}
\]
Thus 
\begin{align*}
h^0(\hX,\sigma^*(kL)-kxF)& \geq h^0(\hX, \sigma^*(kL))-
h^0(kxF,\sigma^*(kL)|_{kxF})\\
& = h^0(X,kL)-h^0(kxF,\cO_{kxF})
\end{align*}
Consequently,
\[
\vol_{\hX}(\sigma^*L-xF)\geq\vol_X(L)-\lim_{k\to\infty}\frac{ h^0(kxF,\cO_{kxF})}{(kx)^n/n!} x^n.
\]
Hence it suffices to show that 
\[
\lim_{j\to\infty}\frac{h^0(jF,\cO_{jF})}{j^n/n!}=\mult_Z X.
\]

We look at the exact sequence
\[
 0\to \cO_{\hX}(-jF)|_{F}\to\cO_{(j+1)F}\to\cO_{jF}\to 0 \numberthis \label{jF}
\]
The isomorphism $F\cong\Proj\bigoplus_{m\geq 0} I_Z^m/I_Z^{m+1}$
yields $\cO_{\hX}(-jF)|_{F}\cong\cO_F(j)$. Since $H^1(F,\cO_F(j))=0$
for $j\gg0$, we have an exact sequence
\[
0\to H^0(F,\cO_F(j))\to H^0((j+1)F,\cO_{(j+1)F})\to H^0(jF,\cO_{jF})\to 0
\]
Hence $h^0((j+1)F,\cO_{(j+1)F})-h^0(jF,\cO_{jF})=h^0(F,\cO_F(j))
=\ell(I_Z^j/I_Z^{j+1})$ for $j\gg 0$. As a result,
\[
 h^0(jF,\cO_{jF})=\ell(\cO_X/I_Z^j)+\const
\]
for $j\gg0$. Hence we prove (1).

\medskip

(2) Assume $\vol_{\hX}(\sigma^*L-xF)=(L^n)-\mult_Z X\cdot x^n$ for some
$x\in\bQ_{>0}$. We will show that 
\begin{equation}\label{cohgrowth}
h^i(\hX,\sigma^*(kL)-kxF)=o(k^n)\quad \textrm{for }i\geq 1,\, k\gg 0
\textrm{ with }kx\in \bZ_{>0}.
\end{equation}
For $i\geq 2$, the long exact sequence \eqref{longexseq} yields
\[
 h^i(\hX,\sigma^*(kL)-kxF)\leq h^{i-1}(kxF,\cO_{kxF})+ h^i(\hX,\sigma^*(kL)).
\]
Since $\sigma^*(kL)$ is nef, we know that $h^i(\hX,\sigma^*(kL))=
o(k^n)$. Serre vanishing theorem implies that $H^{i-1}(F,\cO_F(j))=0$ for $i\geq 2$, $j\gg 0$. Then the exact sequence \eqref{jF} implies that 
\[
 H^{i-1}((j+1)F,\cO_{(j+1)F})\cong H^{i-1}(jF,\cO_{jF})
\]
for $i\geq 2$, $j\gg 0$. As a result, $h^{i-1}(kxF,\cO_{kxF})$ is  
constant for $k\gg 0$. Therefore, $h^i(\hX,\sigma^*(kL)-kxF)=o(k^n)$ for any $i\geq 2$.

By the asymptotic Riemann-Roch theorem we know that
\begin{align*}
 \lim_{k\to\infty}\frac{\chi(\hX,\sigma^*(kL)-kxF)}{k^n/n!} & =((\sigma^*L-xF)^n)\\
 & = (L^n)-\mult_Z X\cdot x^n\\
 & = \vol_{\hX}(\sigma^*L-xF).
\end{align*}
Hence we have
\[
\lim_{k\to\infty}\frac{1}{k^n}\sum_{i=1}^n (-1)^i h^i(\hX,\sigma^*(kL)-kxF)=0.
\]
Since $h^i(\hX,\sigma^*(kL)-kxF)=o(k^n)$ for $i\geq 2$, we conclude that $h^1(\hX,\sigma^*(kL)-kxF)=o(k^n)$ as well.

Now let us prove (2). Denote by $a$ the right-hand side of the equation
in (2). For any nef divisor, its volume is equal to its self intersection
number. Hence it is obvious that $\epsilon_Z(L)\leq a$. Then it suffices to 
show that $a\leq\epsilon_Z(L)$. Equivalently, we only need to show that
for $\epsilon>0$ sufficiently small with $a-\epsilon\in\bQ_{>0}$, 
$\sigma^*L-(a-\epsilon)F$ is ample. Fix $\delta\in\bQ_{>0}$ such that
$\delta<\epsilon_Z(L)$, that is, $\sigma^*L-\delta F$ is ample.
By \cite[Theorem A]{dfkl07}, we only need to show that
\begin{equation}\label{cohgrowth2}
 h^i\big(\hX,m\big((\sigma^*L-(a-\epsilon)F)-t(\sigma^*L-\delta F)\big)
 \big)=o(m^n)
\end{equation}
for any $i>0$, any sufficiently small $t\in\bQ_{>0}$ and suitably
divisible $m$. Denote $b:=(a-\epsilon+t\delta)/(1-t)$.
Notice that
\[
 \frac{1}{1-t}\big((\sigma^*L-(a-\epsilon)F)-t(\sigma^*L-\delta F)\big)
 =\sigma^*L-\frac{a-\epsilon+t\delta}{1-t} F=\sigma^*L-bF.
\]
It is clear that $b<a$ for $t\in\bQ_{>0}$ sufficiently small. Hence from
the description of $a$ we get:
\[
 \vol_{\hX}(\sigma^*L-b F)=(L^n)-\mult_Z X \cdot b^n.
\]
Then applying \eqref{cohgrowth} yields that $h^i(\hX,\sigma^*(kL)-kbF)=o(k^n)$
for $i>0$, sufficiently small $t\in\bQ_{>0}$ and suitably divisible $k$.
In other words, \eqref{cohgrowth2} is true under the same condition for
$i$, $t$ and $m=(1-t)k$. Hence we prove the lemma.
\end{proof}

\subsection{K-semistability and Ding-semistability}\label{KDing}

In this section we recall the definition of K-semistability, and
its recent equivalence Ding-semistability.

\begin{defn}[{\cite{tia97, don02}, see also \cite{lx14}}]\label{defKstable}
Let $X$ be an $n$-dimensional $\bQ$-Fano variety. 
\begin{enumerate}[label=(\alph*)]
\item Let $r\in \bQ_{>0}$ such that $L:=-rK_X$ is Cartier. 
A \textit{test configuration} (resp. a \textit{semi test configuration})
of $(X, L)$ consists of the following data:
\begin{itemize}
\item a variety $\cX$ admitting a $\bG_m$-action and a 
$\bG_m$-equivariant morphism $\pi: \cX\rightarrow \bA^1$, 
where the action of $\bG_m$ on $\bA^1$ is given by the standard 
multiplication;
\item a $\bG_m$-equivariant $\pi$-ample (resp. $\pi$-semiample) line 
bundle $\cL$ on $\cX$ such that $(\cX, \cL)|_{\pi^{-1}
(\bA^1\setminus\{0\})}$ is equivariantly isomorphic to 
$(X, L)\times (\bA^1\setminus\{0\})$ with the natural $\bG_m$-action.
\end{itemize}
Moreover, if $\cX$ is normal in addition, then we say that $(\cX,\cL)/\bA^1$
is a \textit{normal test configuration} (resp. a \textit{normal semi 
test configuration}) of $(X,L)$.

\item 
Assume that $(\cX, \cL)\rightarrow \bA^1$ is a normal test 
configuration. Let $\bar{\pi}: (\bar{\cX}, \bar{\cL})\rightarrow \bP^1$ be the
natural equivariant compactification of $(\cX, \cL)\rightarrow \bA^1$.
The \textit{CM weight} (equivalently, \emph{Donaldson-Futaki invariant}) of $(\cX, \cL)$ is defined by the intersection formula:
\[
\CM(\cX, \cL):=\frac{1}{(n+1)((-K_X)^n)}\left(\frac{n}{r^{n+1}}
(\bar{\cL}^{n+1})+\frac{n+1}{r^n}(\bar{\cL}^n\cdot K_{\bar{\cX}/\bP^1})\right).
\]
\item
\begin{itemize}
\item The pair $(X, -K_X)$ is called \textit{K-semistable} if 
$\CM(\cX, \cL)\ge 0$ for any normal test configuration
$(\cX, \cL)/\bA^1$ of $(X, -rK_X)$.
\item The pair $(X, -K_X)$ is called \textit{K-polystable} if 
$\CM(\cX, \cL)\ge 0$ for any normal test configuration 
$(\cX, \cL)/\bA^1$ of $(X, -rK_X)$, and the equality 
holds if and only if $\cX\cong X\times\bA^1$.
\end{itemize}
\end{enumerate}
\end{defn}

In \cite{din88}, Ding introduced the so-called \emph{Ding
functional} on the space of K\"ahler metrics on Fano manifolds and showed that 
K\"ahler-Einstein metrics are critical points of the Ding functional.
Recently, Berman \cite{ber12} established a formula expressing CM-weights
in terms of the slope of Ding functional along a geodesic ray in the
space of all bounded positively curved metrics on the anti-canonical
$\bQ$-line bundle of a $\bQ$-Fano variety. Later on, this slope was called \emph{Ding invariant}
by Fujita \cite{fuj15} in order to define the concept of \emph{Ding-semistability}.
Next we will recall the recent work by Berman \cite{ber12} and Fujita \cite{fuj15}.
\begin{defn}[\cite{ber12, fuj15}] \label{ding}
Let $X$ be an $n$-dimensional $\bQ$-Fano variety.
\begin{enumerate}[label=(\alph*)]
\item
Let $(\cX, \cL)/\bA^1$ be a normal semi test configuration 
of $(X, -rK_X)$ and $(\bar{\cX}, \bar{\cL})/\bP^1$ be its natural
compactification. Let $D_{(\cX, \cL)}$ be the $\bQ$-divisor on $\cX$
satisfying the following conditions:
\begin{itemize}
\item The support ${\rm Supp}D_{(\cX, \cL)}$ is contained in $\cX_0$. 
\item The divisor $-r D_{(\cX, \cL)}$ is a $\bZ$-divisor corresponding
to the divisorial sheaf $\bar{\cL}(r K_{\bar{\cX}/\bP^1})$.
\end{itemize}
\item 
The \textit{Ding invariant} $\Ding(\cX, \cL)$ of $(\cX, \cL)/\bA^1$ is defined
as
\[
\Ding(\cX, \cL):=\frac{-(\bar{\cL}^{n+1})}{(n+1)r^{n+1}((-K_X)^n)}
-\left(1-\lct(\cX, D_{(\cX,\cL)}; \cX_0)\right).
\]
Here $\lct(\cX, D_{(\cX,\cL)}; \cX_0)$ is defined in the sense of \eqref{sublct}.
\item
$X$ is called \textit{Ding-semistable} if $\Ding(\cX, \cL)\ge 0$ for 
any normal test configuration $(\cX, \cL)/\bA^1$ of $(X, -rK_X)$.
\end{enumerate}
\end{defn}

The following theorem illustrates relations among the existence of K\"ahler-Einstein metrics,
K-semistability and Ding-semistability of $\bQ$-Fano varieties. 

\begin{thm}\label{DingKE}
\begin{enumerate}
 \item If a $\bQ$-Fano variety $X$ admits a K\"ahler-Einstein
metric, then $X$ is Ding-semistable.
 \item A $\bQ$-Fano variety is K-semistable if and only if it is Ding-semistable.
\end{enumerate}
\end{thm}

\begin{proof}
 (1) is proved by Berman in \cite{ber12} (see also \cite[Theorem 3.2]{fuj15}).
 
 (2) The direction that Ding-semistability implies K-semistability
 is proved by Berman in \cite{ber12} using the inequality that 
 $\Ding(\cX,\cL)\geq \CM(\cX,\cL)$ for any normal test configuration
 $(\cX,\cL)$ (see also \cite[Theorem 3.2.2]{fuj15}). The direction that
 K-semistability implies Ding-semistability is proved in \cite{bbj15}
 based on \cite{lx14} (see also \cite[Section 3]{fuj16}). 
\end{proof}

\subsection{Filtrations}

The use of filtrations to study stability notions was initiated by Witt
Nystr\"om in \cite{wn12}.
We recall the relevant definitions about filtrations after \cite{bc11, wn12}
(see also \cite{bhj15} and \cite[Section 4.1]{fuj15}).
\begin{defn}\label{defil}
 A \textit{good filtration} of a graded $\bC$-algebra 
 $S=\bigoplus_{m=0}^{\infty}S_m$ is a decreasing, left continuous, multiplicative
and linearly bounded $\bR$-filtrations of $S$. In other words, for each $m\ge 0\in \bZ$, there is a family of subspaces $\{\cF^{x} S_m\}_{x\in \bR}$ of $S_m$ such that:
\begin{itemize}
\item $\cF^x S_m\subseteq \cF^{x'}S_m$, if $x\ge x'$;
\item $\cF^x S_m=\bigcap_{x'<x}\cF^{x'} S_m$; 
\item $\cF^x S_m\cdot \cF^{x'} S_{m'}\subseteq \cF^{x+x'} S_{m+m'}$, for any $x, x'\in \bR$ and $m, m'\in \bZ_{\ge 0}$;
\item $e_{\min}(\cF)>-\infty$ and $e_{\max}(\cF)<+\infty$, where $e_{\min}(\cF)$ and $e_{\max}(\cF)$ are defined by the following 
operations:

\begin{equation}
\def\arraystretch{1.5}
\begin{array}{l}
e_{\min}(S_m,\cF):=\inf\{t\in\bR: \cF^t S_m\neq S_m\}; \\
e_{\max}(S_m,\cF):=\sup\{t\in\bR: \cF^t S_m\neq 0\};\\
\displaystyle e_{\min}(\cF)=e_{\min}(S_{\bullet}, \cF):=
\liminf_{m\rightarrow \infty} \frac{e_{\min}(S_m, \cF)}{m}; \\
\displaystyle e_{\max}(\cF)=e_{\max}(S_{\bullet}, \cF):=
\limsup_{i\rightarrow \infty} \frac{e_{\max}(S_m, \cF)}{m}. 
\end{array}
\end{equation}

\end{itemize}
\end{defn}

Define $S^{(t)}:=\bigoplus_{k=0}^{\infty} \cF^{kt} S_k$. When we want to emphasize the dependence of $ S^{(t)}$ on the filtration $\cF$, we also denote $ S^{(t)}$ by $\cF S^{(t)}$.
The following concept of volume will be important for us:
\begin{equation}
\vol\left( S^{(t)}\right)=\vol\left(\cF S^{(t)}\right):=
\limsup_{k\rightarrow\infty}\frac{\dim_{\bC}\cF^{mt}S_m}{m^{n}/n!}.
\end{equation}

Now assume $L$ is an ample line bundle over $X$ and 
$S=\bigoplus_{m=0}^{\infty} H^0(X, L^m)=\bigoplus_{m=0}^{\infty} S_m$
is the section ring of $(X, L)$. Then following \cite{fuj15}, 
we can define a sequence of ideal sheaves on $X$:
\begin{equation}\label{filtideal}
I^{\cF}_{(m,x)}={\rm Image}\left(\cF^xS_m\otimes L^{-m}\rightarrow \cO_X\right), %=\cF^{x}S_m
\end{equation}
and define $\overline{\cF}^x S_m:=H^0(V, L^m\cdot I^{\cF}_{(m,x)})$ to be the saturation of $\cF^x S_m$ such that
$\cF^x S^m\subseteq \overline{\cF}^x S_m$.
$\cF$ is called \textit{saturated} if $\overline{\cF}^x S_m=\cF^x S_m$ for any $x\in \bR$ and $m\in \bZ_{\ge 0}$. 
Notice that with the notation above we have:
\[
\vol\left(\overline{\cF} S^{(t)}\right):=\limsup_{k\rightarrow\infty}
\frac{\dim_{\bC}\overline{\cF}^{kt}H^0(X, kL)}{k^n/n!}.
\]
\subsection{Normalized volume of real valuations}\label{secval}
Let $(X,p)$ be a normal $\bQ$-Gorenstein singularity with $\dim X=n$.
A \textit{real valuation} $v$ on the function field $\bC(X)$ is a map
$v:\bC(X)\to\bR$, satisfying:
\begin{itemize}
 \item $v(fg)=v(f)+v(g)$;
 \item $v(f+g)\geq \min\{v(f),v(g)\}$;
 \item $v(\bC^*)=0$.
\end{itemize}

Denote by $\cO_v:=\{f\in\bC(X): v(f)\geq 0\}$ the \textit{valuation ring}
of $v$. Denote by $(\cO_p,\fm_p)$ the local ring of $X$ at $p$. 
We say that a real valuation $v$ on $\bC(X)$ is \textit{centered at}
$p$ if the local ring $\cO_p$ is dominated by $\cO_v$. In other words,
$v$ is nonnegative on
$\cO_p$ and strictly positive on $\fm_p$.
Denote by $\Val_{X,p}$ the set of real valuations centered at $p$. Given a valuation $v\in\Val_{X,p}$, define the valuation ideals for any $x\in\bR$ as follows:
\[
\fa_x(v):=\{f\in\cO_X: v(f)\geq x\}.
\]
The \textit{volume} of the valuation $v$ is defined as
\[
\vol(v):=\lim_{x\to+\infty}\frac{\ell(\cO_X/\fa_x(v))}{x^n/n!}.
\]
By \cite{els03, mus02, cut13}, the limit on the right hand side exists and is equal to the multiplicity of the graded family of ideals
$\fa_\bullet(v)$:
\[
\vol(v)=\lim_{m\to\infty}\frac{\mult(\fa_m(v))}{m^n}=:\mult(\fa_\bullet(v)).
\]

Following \cite{jm12,bdffu}, we can define the \textit{log discrepancy}
for any valuation $v\in\Val_{X}$. This is achieved in three steps in
\cite{jm12,bdffu}. Firstly, for a divisorial valuation $\ord_E$ associated
to a prime divisor $E$ over $X$, define $A_X(E):=1+\ord_E(K_{Y/X})$,
where $\pi:Y\to X$ is a resolution of $X$ containing $E$.
Next we define the log discrepancy of a quasi-monomial valuation (also called Abhyankar valuation).
For a resolution $\pi:Y\to X$, let $\underline{y}:=(y_1,\cdots,y_r)$
be a system of algebraic coordinate of a point $\eta\in Y$. By
\cite[Proposition 3.1]{jm12}, for every $\alpha=(\alpha_1,\cdots,\alpha_r)\in\bR_{\geq 0}^n$
one can associate a unique valuation $\val_{\alpha}=\val_{\underline{y},\alpha}\in\Val_X$
with the following property: whenever $f\in\cO_{Y,\eta}$ is written 
in $\widehat{\cO_{Y,\eta}}$ as $f=\sum_{\beta\in\bZ_{\geq 0}^r}c_\beta y^{\beta}$,
with each $c_\beta\in\widehat{\cO_{Y,\eta}}$ either zero or a unit, we have
\[
\val_{\alpha}(f)=\min\{\langle\alpha,\beta\rangle\mid c_\beta\neq 0\}.
\]
Such valuations are called \emph{quasi-monomial valuations} (or 
equivalently, \emph{Abhyankar valuations} by \cite{els03}).
Let $(Y,D=\sum_{k=1}^N D_k)$ be a log smooth model of $X$, i.e. $\pi:Y\to X$
is an isomorphism outside of the support of $D$.
We denote by $\QM_\eta(Y,D)$ the set of all quasi-monomial valuations $v$ that can be
described at the point $\eta\in Y$ with respect to coordinates $(y_1,\cdots,y_r)$ such that each $y_i$
defines at $\eta$ an irreducible component of $D$ (hence $\eta$ is the generic point of a connected
component of the intersection of some of the $D_i$'s). We put $\QM(Y,D):=\bigcup_{\eta}\QM(Y,D)$.
Suppose $\eta$ is a generic point of a connected component of
$D_{i_1}\cap\cdots\cap D_{i_r}$, then the log discrepancy of $\val_\alpha$
is defined as
\[
 A_X(\val_{\alpha}):=\sum_{j=1}^r\alpha_j\cdot A_X(\ord_{D_{i_j}})
 =\sum_{j=1}^r \alpha_j\cdot(1+\ord_{D_{i_j}}(K_{Y/X})).
\]
Finally, in \cite{jm12} Jonsson and Musta\c{t}\u{a} showed that there 
exists a retraction map $r_{Y,D}:\Val_X\to\QM(Y,D)$ for any log smooth model
$(Y,D)$ over $X$, such that it induces a homeomorphism $\Val_X\to\varprojlim_{(Y,D)}\mathrm{QM}(Y,D)$.
For any real valuation $v\in\Val_X$, we define 
\[
 A_X(v)=\sup_{(Y,D)} A_X(r_{(Y,D)}(v))
\]
where $(Y,D)$ ranges over all log smooth models over $X$.
For details, see \cite{jm12} and \cite[Theorem 3.1]{bdffu}. It is possible that $A_X(v)=+\infty$ for some $v\in \Val_X$,
see e.g. \cite[Remark 5.12]{jm12}.

Following \cite{li15a}, the \textit{normalized volume function}
$\hvol(v)$ for any
$v\in\Val_{X,p}$ is defined as
\[
\hvol(v):=\begin{cases}
           A_X(v)^n\cdot\vol(v), & \textrm{if }A_X(v)<+\infty;\\
           +\infty, &\textrm{if }A_X(v)=+\infty.
          \end{cases}
\]
Notice that $\hvol(v)$ is rescaling invariant: $\hvol(\lambda v)=\hvol(v)$
for any $\lambda>0$.
By Izumi's theorem (see \cite{izu85,ree87,els03,jm12,bfj14,li15a}), one can show that $\hvol$ is uniformly 
bounded from below by a positive number on $\Val_{X,p}$.

\section{Proofs of main theorems}

\subsection{Proofs}\label{pf}

It was proved by Berman \cite{ber12} that the existence of a K\"ahler-Einstein metric
on a $\bQ$-Fano variety implies Ding-semistability (see Theorem \ref{DingKE}).
In this section, we will prove Theorem \ref{mainthm} and \ref{mainthm2} with the 
weaker assumption that $X$ is Ding-semistable.

The following result by Fujita is crucial for proving Theorem \ref{mainthm}. 
\begin{thm}[{\cite[Theorem 1.2]{fuj15}}]\label{beta}
Let $X$ be a $\bQ$-Fano variety. Assume that $X$ is Ding-semistable. Take any nonempty proper closed subscheme $\emptyset\neq Z\subsetneq X$ corresponds to an ideal sheaf $0\neq I_Z\subsetneq \cO_X$. Let $\sigma:\hX\to X$ be the blow up along $Z$, let $F\subset \hX$ be the Cartier divisor defined by the equation $\cO_{\hX}(-F)=\sigma^{-1}I_Z\cdot\cO_{\hX}$. Then we have $\beta(Z)\geq 0$,
where
\[
\beta(Z):=\lct(X;I_Z)\cdot\vol_X(-K_X)-\int_0^{+\infty}
\vol_{\hX}(\sigma^*(-K_X)-xF)dx.
\]
\end{thm}

We will prove the following theorem, a stronger result that implies Theorem \ref{mainthm}.

\begin{thm}\label{main1}
Let $X$ be a Ding-semistable $\bQ$-Fano variety. 
Let $p\in X$ be a 
closed point. Let $Z$ be a closed subscheme of $X$ with $\Supp~Z=\{p\}$.
Then we have 
\begin{equation}\label{firstvol}
 ((-K_X)^n)\leq \left(1+\frac{1}{n}\right)^n\lct(X;I_Z)^n
 \mult_Z X.
\end{equation}
\end{thm}

\begin{proof}
Let $\sigma:\hX\to X$ be the blow up along $Z$, 
and $F\subset\hX$ be the Cartier divisor defined by the 
equation $\cO_{\hX}(-F)=\sigma^{-1}I_Z\cdot\cO_{\hX}$.
By Theorem \ref{beta}, we have
\[
\lct(X;I_Z)\cdot((-K_X)^n)\geq
\int_0^{+\infty} \vol_{\hX}(\sigma^*(-K_X)-xF)dx.
\]
On the other hand, by Lemma \ref{volume}, we have
\begin{align*}
\int_0^{+\infty} \vol_{\hX}(\sigma^*(-K_X)-xF)dx & \geq 
\int_0^{\sqrt[n]{((-K_X)^n)/\mult_Z X}} \big(((-K_X)^n)-\mult_Z X\cdot x^n\big) dx\\
& = \frac{n}{n+1}((-K_X))^n\cdot\sqrt[n]{((-K_X)^n)/\mult_Z X}
\end{align*}
Hence we get the desired inequality
\[
 ((-K_X)^n)\leq \left(1+\frac{1}{n}\right)^n\lct(X;I_Z)^n
 \mult_Z X.
\]
\end{proof}

\begin{rem}\label{volequality}
 As we see in the proof above, if the equality of \eqref{firstvol} holds,
 then 
 \[
 \vol_{\hX}(\sigma_*(-K_X)-xF)=\max\{((-K_X)^n)-\mult_Z X\cdot x^n,0\}
 \]
 for any $x\in \bR_{\geq 0}$. Then by Lemma \ref{volume} (2) we have that
 $\epsilon_Z(-K_X)=(1+\frac{1}{n})\lct(X;I_Z)$.
\end{rem}
\medskip

The following result by Fujita is the main tool to prove Theorem \ref{mainthm2}.

\begin{thm}[{\cite[Theorem 4.9]{fuj15}}]\label{fuj4.9}
Let $X$ be a $\bQ$-Fano variety of dimension $n$.
Assume that $X$ is Ding-semistable. Let $\cF$ be a good graded
filtration of $S=\bigoplus_{m=0}^{+\infty} H^0(X, mL)$ where 
$L=-rK_X$ is Cartier. Then the pair $(X \times\bA^1, 
\cI_\bullet^{\cdot (1/r)} \cdot (t)^{\cdot d_\infty})$ is sub log canonical
in the sense of \cite[Definition 2.4]{fuj15}, where
\begin{eqnarray*}
&&\mathcal{I}_m=I^{\cF}_{(m, m e_{+})}+I^{\cF}_{(m, m e_{+}-1)}t^1+\cdots+I^{\cF}_{(m, m e_{-}+1)} t^{m(e_+-e_-)-1}+(t^{m(e_+-e_-)}),\\
&&d_{\infty}=1-\frac{e_+-e_-}{r}+\frac{1}{r^{n+1}((-K_X)^n)}\int_{e_-}^{e_+}\vol\left(\overline{\cF} S^{(t)}\right)dt, \\
\end{eqnarray*}
and $e_+, e_-\in \bZ$ with $e_+\ge e_{\max}(S_\bullet, \cF)$ and $e_-\le e_{\min}(S_\bullet, \cF)$.
\end{thm}

We will follow the notation of Theorem \ref{fuj4.9}.
Let $v\in\Val_{X,p}$ be a real valuation centered at $p$.
Define an $\bR$-filtration $\cF_v$ of $S$ as
\[
\cF_v^x S_m:=H^0(X, L^m\cdot\fa_{x}(v)).
\]

\begin{lem}\label{goodfil}
The $\bR$-filtration $\cF_v$ of $S$ is decreasing, left-continuous,
multiplicative and saturated. Moreover, if $A_X(v)<+\infty$ then
$\cF_v$ is also linearly bounded.
\end{lem}

\begin{proof}
The decreasing, left-continuous and multiplicative properties of 
$\cF_v$ follows from the corresponding properties 
of $\{\fa_x(v)\}_{x\in\bR}$. 
To prove that $\cF_v$ is saturated, notice that the homomorphism
\[
 \cF_v^x S_m\otimes_{\bC}L^{-m}\twoheadrightarrow I_{(m,x)}^{\cF_v}
\]
induces the inclusion $I_{(m,x)}^{\cF_v}\subset\fa_x(v)$ for any
$x\in\bR$. Thus $\overline{\cF}^x S_m= H^0(X,L^m\cdot I_{(m,x)}^{\cF_v})
\subset\cF^x S_m$, which implies that $\cF_v$ is saturated.

If $A_X(v)<+\infty$, then by \cite[Theorem 3.1]{li15a} there exists a constant
$c=c(X,p)$ such that $v\leq cA_X(v)\ord_p$. So it is easy to see that
\[
e_{\max}(S_\bullet,\cF_v)\leq cA_X(v)\cdot e_{\max}(S_\bullet,\cF_{\ord_p})<+\infty.
\]
The second inequality follows from the work in \cite{bkms14}.
Next, it is clear that $\fa_x(v)=\cO_X$ for $x\leq 0$. Hence 
$\cF_v^x S_m=S_m$ for $x\leq 0$, which implies
$e_{\min}(S_\bullet,\cF_v)\geq 0$. Hence $\cF_v$ is linearly bounded.
\end{proof}

It is clear that $\cF_v^x S_m=S_m$ for $x\leq 0$. Hence we may choose
$e_-=0$. The graded sequence of ideal sheaves $\cI_\bullet$ on 
$X\times\bA^1$ becomes:
\[
 \cI_m=I_{(m,me_+)}^{\cF_v}+I_{(m,me_+-1)}^{\cF_v}t^1+\cdots+
 I_{(m,1)}^{\cF_v}t^{me_+ -1}+(t^{me_+}).
\]
We also have that 
\[
 d_{\infty}=1-\frac{e_+}{r}+\frac{1}{r^{n+1}((-K_X)^n)}\int_{0}^{+\infty}
 \vol\left(\cF_v S^{(t)}\right)dt.
\]

The valuation $v$ extends to a $\bG_m$-invariant valuation $\bar{v}$
on $\bC(X\times\bA^1)=\bC(X)(t)$, such that for any 
$f=\sum_j f_j t^j$ we have
\begin{equation}\label{vbar}
\bar{v}(f)=\min_j\{v(f_j)+j\}.
\end{equation}

\begin{prop}\label{val1}
Let $X$ be a $\bQ$-Fano variety of dimension $n$.
Assume that $X$ is Ding-semistable. Let $p\in X$ be a 
closed point. Let $v\in \Val_{X,p}$ be a real valuation centered at 
$p$.
Then we have 
\begin{equation}\label{logFujval}
 A_X(v)\geq\frac{1}{r^{n+1}((-K_X)^n)}\int_0^{+\infty}
 \vol\left(\cF_v S^{(t)}\right)dt.
\end{equation}
\end{prop}

\begin{proof}
We may assume that $A_X(v)<+\infty$, since otherwise the inequality holds
automatically. Hence Lemma \ref{goodfil} implies that $\cF_v$ is good and saturated.

By Theorem \ref{fuj4.9}, we know that $(X\times\bA^1, \cI_\bullet^{\cdot(1/r)}\cdot (t)^{\cdot d_\infty})$ is sub log canonical.
Since $X\times\bA^1$ has klt singularities,
for any $0<\epsilon\ll 1$ there exists $m=m(\epsilon)$ such that
\[
 \cO_{X\times\bA^1}\subset\cJ(X\times\bA^1,
 \cI_m^{\cdot(1-\epsilon)/(rm)}\cdot (t)^{\cdot(1-\epsilon)d_\infty}),
\]
where the right hand side is the multiplier ideal sheaf.
By \cite[Theorem 1.2]{bdffu} we know that the following inequality holds for any real
valuation $u$ on $X\times\bA^1$:
\begin{equation}
 A_{X\times\bA^1}(u)>\frac{(1-\epsilon)}{rm}u(\cI_m)+(1-\epsilon)d_\infty u(t).
\end{equation}
Let us choose a sequence of quasi-monomial real valuations $\{v_n\}$ on $X$
such that $v_n\to v$ and $A_X(v_n)\to A_X(v)$ when $n\to\infty$. It is easy
to see that $\{\bar{v}_n\}$ is a sequence of quasi-monimial valuations on
$X\times\bA^1$ satisfying $\bar{v}_n(t)=1$ and 
\[
 A_{X\times\bA^1}(\bar{v}_n)=A_X(v_n)+1.
\]
Hence we have
\begin{align*}
A_X(v)+1  &=\lim_{n\to\infty} A_X(\bar{v}_n)\\
& \geq \frac{(1-\epsilon)}{rm}\lim_{n\to\infty}\bar{v}_n(\cI_m) + (1-\epsilon)d_\infty.
\end{align*}
From the definition of $\cF_v^x S_m$ we get:
\[
v\left(I^{\cF_v}_{(m,x)}\right)\ge x.
\]
Therefore,
\begin{align*}
 \lim_{n\to\infty}\bar{v}_n(\cI_m) & =\lim_{n\to\infty} 
 \min_{0\leq j\leq re_+}\left\{v_n\left(I^{\cF_v}_{(m,j)}\right)+ me_+ -j\right\}\\
 & =\min_{0\leq j\leq re_+}\left\{\lim_{n\to\infty}v_n\left(I^{\cF_v}_{(m,j)}\right)+ me_+ -j\right\}\\
 & =\min_{0\leq j\leq re_+}\left\{v\left(I^{\cF_v}_{(m,j)}\right)+me_+ - j\right\}\\
 & \geq me_+
\end{align*}
Hence when $\epsilon\to 0+$ we get:
\begin{equation}
 A_X(v)+1\geq \frac{e_+}{r} + d_\infty.
\end{equation}
Therefore,
\[
 A_X(v)\geq -1 + \frac{e_+}{r} + d_\infty = \frac{1}{r^{n+1}((-K_X)^n)}\int_0^{+\infty}
 \vol\left(\cF_v S^{(t)}\right)dt.
\]
Hence we get the desired inequality.
\end{proof}

The following theorem is a stronger result that implies Theorem \ref{mainthm2}.

\begin{thm}\label{main2}
Let $X$ be a Ding-semistable $\bQ$-Fano variety. 
Let $p\in X$ be a 
closed point. Let $v\in \Val_{X,p}$ be a real valuation centered at $p$.
Then we have 
\[
 ((-K_X)^n)\leq \left(1+\frac{1}{n}\right)^n\hvol(v).
\]
\end{thm}

\begin{proof}
We may assume $A_X(v)<+\infty$, since otherwise the inequality holds automatically.
Since $\cF_v^{mx} S_m=H^0(X,L^m\cdot\fa_{mx}(v))$, we have the exact sequence
\[
 0\to \cF_v^{mx}S_m\to H^0(X,L^{m})\to H^0(X,L^{m}\otimes(\cO_X/\fa_{mx}(v))).
\]
Hence we have
\[
 \dim \cF_v^{mx}S_m\geq h^0(X,L^{m})-\ell(\cO_X/\fa_{mx}(v)).
\]
This implies
\begin{equation}\label{liuest}
\vol(\cF_v S^{(x)})\ge (L^n)-\vol(v)x^n.
\end{equation}
%Now we can prove our result following the method of \cite{Fuj15, Liu16}:
So we have the estimate:
\begin{align*}
\int^{+\infty}_0\vol\left(\cF_v S^{(x)}\right)dx &\ge \int_0^{\sqrt[n]{(L^n)/\vol(v)}}
\left((L^n)-\vol(v)x^n\right)dx\\
&=\frac{n}{n+1}(L^n)\cdot \sqrt[n]{(L^n)/\vol(v)}.
\end{align*}
Applying \eqref{logFujval} we get the inequality:
\begin{eqnarray*}
A_X(v)&\ge&\frac{1}{r(L^{n})}\int^{+\infty}_0\vol\left(\cF_v S^{(t)}\right)dt\\
&\ge& \frac{1}{r(L^n)}\frac{n}{n+1} (L^{n})\sqrt[n]{(L^n)/\vol(v)}
=\frac{n}{n+1}\sqrt[n]{((-K_X)^n)/\vol(v)}.
\end{eqnarray*}
This is exactly equivalent to
\[
 ((-K_X)^n)\leq \left(1+\frac{1}{n}\right)^n A_X(v)^n\vol(v)=\left(1+\frac{1}{n}\right)^n\hvol(v).
\]
\end{proof}

\subsection{Applications}\label{app}

The following result is an easy consequence of Theorem \ref{mainthm}.
\begin{thm}[\cite{dfem04}]\label{dfem}
Let $X$ be a variety of dimension $n$.
Let $Z$ be a closed subscheme of $X$ supported at a single smooth closed
point. Then
\[
\left(\frac{\lct(X;I_Z)}{n}\right)^n\mult_Z X\geq 1.
\]
\end{thm}

\begin{proof}
 It is clear that the multiplicity is preserved under completion. Next, \cite[Proposition 2.11]{dfem11} implies that $\lct$ is also preserved under completion.
 Hence it suffices to prove the theorem for $X= \bP^n$. This follows 
 directly by applying Theorem \ref{mainthm} to the case $X=\bP^n$.
\end{proof}

Now we will illustrate some volume bounds for singular $X$.

\begin{thm}\label{quot1}
Let $X$ be a Ding-semistable $\bQ$-Fano variety. Let $p\in X$ be a
closed point. Suppose $(X,p)$ is a quotient singularity with local
analytic model $\bC^2/G$ where $G\subset GL(n,\bC)$ acts freely
in codimension $1$. Then
\[
 ((-K_X)^n)\leq \frac{(n+1)^n}{|G|}.
\]
\end{thm}

\begin{proof}
 Let $(Y,o):=(\bC^n/G, 0)$ be the local analytic model of the quotient 
 singularity $(X,p)$. 
 Let $H:= G\cap \mathbb{G}_m$ with $d:=|H|$. 
 Define the real valuation $u_0$ on $Y$ to be the pushforward
 of the valuation $\ord_0$ on $\bC^n$ under the quotient map 
 $\bC^n\to Y$. 
 
 We first show that $\hvol(u_0)=\frac{n^n}{|G|}$.
 Let $\widehat{\bC^n}$ be the blow up of $\bC^n$ at the origin $0$ 
 with exceptional divisor $E$. Denote by $\pi:\bC^n\to Y$
 the quotient map. Then $\pi$ lifts to $\widehat{\bC^n}$ as 
 $\hat{\pi}:\widehat{\bC^n}\to \hat{Y}$, where
 $\hat{Y}:=\widehat{\bC^n}/G$. We have the following commutative diagram:
 \[
 \begin{tikzcd}
  \widehat{\bC^n}\arrow{r}{\hat{\pi}}\arrow{d}[swap]{g} & 
  \hat{Y}
  \arrow{d}{h}\\
  \bC^n\arrow{r}{\pi} & Y
 \end{tikzcd}
 \]
 Let $F\subset\hat{Y}$ be the exceptional divisor of $h$. For a 
 general point on $F$, its stabilizer under the $G$-action is 
 exactly $H$. So $\hat{\pi}^*F=d E$, which implies that $u_0=
 \pi_*(\ord_E)=d~\ord_F$. It is clear that
 \[
  K_{\widehat{\bC^n}}=\hat{\pi}^*\left(K_{\hat{Y}}+\left(1-\frac{1}{d}\right) F \right).
 \]
 Then combining these equalities with $K_{\widehat{\bC^n}}=g^* 
 K_{\bC^n}+(n-1)E=g^*(\pi^* K_Y)+(n-1)E$, we get
 \[
  K_{\hat{Y}}=h^*K_{Y} +\left(\frac{n}{d}-1\right) F.
 \]
 Hence $A_X(\ord_F)=\frac{n}{d}$ and $A_X(u_0)=d\cdot A_X(\ord_F)=n$. 
 Lemma \ref{quotmult} implies that $\vol(u_0)=\frac{1}{|G|}$.
 Therefore, $\hvol(u_0)=\frac{n^n}{|G|}$.
 
 Since $A_Y(u_0)<+\infty$, a singular version of \cite[Corollary 5.11]{jm12}
 for klt singularities (such a statement is true thanks to the Izumi inequality
 for klt singularities \cite[Proposition 3.1]{li15a})
 implies that $u_0$ has a unique extension, say $\hat{u}_0$,
 as a valuation on $\Spec\widehat{\cO_{Y,o}}$ that centered at the closed
 point. By assumption we have an isomorphism $\phi:\widehat{\cO_{X,p}}\to 
 \widehat{\cO_{Y,o}}$. Let $v_0:= (\phi^*\hat{u}_0)|_{\cO_{X,p}}$, then
 $v_0\in\Val_{X,p}$. By a singular version of \cite[Proposition 5.13]{jm12}, we know that
 $A_X(v_0)=A_Y(u_0)$. On the other hand, since the colength of any
 $\fm$-primary ideal does not change under completion, we have 
 $\vol(v_0)=\vol(\hat{u}_0)=\vol(u_0)$. As a result, we have
 \[
  \hvol(v_0)= \hvol(u_0) = \frac{n^n}{|G|}.
 \]
 Hence the theorem follows as a direct application of Theorem \ref{mainthm2}
 to $v=v_0$. 
\end{proof}

\begin{lem}\label{quotmult}
Let $G\subset GL(n,\bC)$ be a finite group acting on $\bC^n=\Spec~\bC[x_1,\cdots,x_n]$. 
Then we have
\[
\lim_{m\to\infty}\frac{\dim_{\bC}\bC[x_1,\cdots,x_n]^G_{< m}}{m^n/n!}=\frac{1}{|G|}.
\]
\end{lem}

\begin{proof}
Denote $W:=\bC[x_1,\cdots,x_n]=\bC x_1\oplus\cdots\oplus\bC x_n$.
Then we have that
\[
 \bC[x_1,\cdots,x_n]\cong\bigoplus_{m\geq 0}\Sym^m W.
\]
Denote by $\rho_m:G\to GL(\Sym^m W)$ the induced representation of 
$G$ on $\Sym^m W$. Since $G$ is finite, $\rho_m(g)$ is diagonalizable
for any $m\geq 0$ and $g\in G$. Denote the $n$ eigenvalues of $\rho_1(g)$
by $\lambda_{g,1},\cdots,\lambda_{g,n}$. Then the eigenvalues of 
$\rho_m(g)$ are exactly all monomials in $\lambda_{g,i}$ of degree
$m$. Let $c_m:=\dim_{\bC}(\Sym^m W)^G$ and 
$d_m:=\dim_{\bC}\bC[x_1,\cdots,x_n]_{\leq m}^G$. From representation
theory we know that
\[
 c_m=\frac{1}{|G|}\sum_{g\in G}\tr(\rho_m(g)).
\]
Hence
\[
 c(t):=\sum_{m=0}^\infty c_m t^m =\frac{1}{|G|}\sum_{g\in G}\frac{1}
 {(1-\lambda_{g,1}t)\cdots (1-\lambda_{g,n}t)}.
\]
This is also known as Molien's theorem \cite{mol1897}.
Since $d_m=\sum_{i=0}^m c_i$, we have
\begin{align*}
 d(t) & :=\sum_{m=0}^\infty d_m t^m = \frac{c(t)}{1-t}\\
  & =\frac{1}{|G|}\sum_{g\in G}\frac{1}
 {(1-t)(1-\lambda_{g,1}t)\cdots (1-\lambda_{g,n}t)}.
\end{align*}
Here all eigenvalues $\lambda_{g,i}$ are roots of unity. Using
partial fraction decomposition, 
for any $g\neq id$ we have  
\[
 \frac{1}
 {(1-t)(1-\lambda_{g,1}t)\cdots (1-\lambda_{g,n}t)}=\sum_{m=0}^\infty
 o(m^n)t^m.
\]
Hence 
\begin{align*}
 d(t)& =\frac{1}{|G|}\cdot\frac{1}{(1-t)^{n+1}} + \sum_{m=0}^\infty
 o(m^n)t^m\\
 &=\sum_{m=0}^\infty\left(\frac{1}{|G|}\binom{n+m}{n}+o(m^n)\right)t^m.
\end{align*}
As a result,
\[
 \lim_{m\to\infty}\frac{d_{m-1}}{m^n/n!}=\lim_{m\to\infty}
 \frac{\frac{1}{|G|}\binom{n+m-1}{n}}{m^n/n!}=\frac{1}{|G|}.
\]
\end{proof}

\begin{thm}\label{nonterm}
Let $X$ be a Ding-semistable $\bQ$-Fano variety. Let $p\in X$ be an isolated singularity.
If $X$ is not terminal at $p$, then 
\[
((-K_X)^n)\leq \left(1+\frac{1}{n}\right)^n \mult_p X< e~\mult_p X.
\]
\end{thm}

\begin{proof}
 Denote by $\fm_p$ the maximal ideal at $p$.
 We first show that $\lct(X;\fm_p)\leq 1$. Since $p$ is an isolated singularity,
 we may take a log resolution of $(X,\fm_p)$, namely $\pi:Y\to X$ such
 that $\pi$ is an isomorphism away from $p$ and $\pi^{-1}\fm_p\cdot\cO_Y$
 is an invertible ideal sheaf on $Y$. Let $E_i$ be the exceptional divisors
 of $\pi$. We define the numbers $a_i$ and $b_i$ by
 \[
 K_Y=\pi^*K_X+\sum_i a_i E_i\quad\textrm{and}\quad \pi^{-1}\fm_p\cdot\cO_Y=
 \cO_Y(-\sum_i b_i E_i).
 \]
 It is clear that $\lct(X;\fm_p)=\min_i\frac{1+a_i}{b_i}$. Since $\pi$ is an isomorphism away from $p$, we have $b_i\geq 1$ for
 any $i$. Since $X$ is not terminal at $p$, there exists an index $i_0$ such that 
 $a_{i_0}\leq 0$.
 Hence 
 \[
  \lct(X; \fm_p)\leq \frac{1+a_{i_0}}{b_{i_0}}\leq 1.
 \]
 So we finish the proof by applying Theorem \ref{mainthm} to $Z=p$.
\end{proof}

\section{Comparing invariants for ideals and valuations}\label{seccompare}

In this section, we will always assume that $(X,p)$ is a normal
$\bQ$-Gorenstein klt singularity of dimension $n$ with $p\in X$ 
a closed point.

\subsection{The infimums of invariants}
\begin{lem}\label{compare1}
Let $\fa$ be an ideal sheaf on $X$ supported at $p$. 
Let $v_0\in \Val_{X}$ be a divisorial valuation that computes 
$\lct(\fa)$. Then $v_0$ is centered at $p$, and we have
\[
\lct(\fa)^n\cdot\mult(\fa)\geq \hvol(v_0).
\]
\end{lem}

\begin{proof}
Since $v_0$ computes $\lct(\fa)$, we have that $\lct(\fa)=A_X(v_0)/v_0(\fa)$. Denote $\alpha:=v_0(\fa)$. Then $\fa^m\subset \fa_{m\alpha}(v_0)$ since $v_0(\fa^m)=m\alpha$. Therefore,
\begin{align*}
 \mult(\fa) & =\lim_{m\to\infty}\frac{\ell(\cO_X/\fa^m)}{m^n/n!}\\
 & \geq \lim_{m\to\infty}\frac{\ell(\cO_X/\fa_{m\alpha}(v_0))}{m^n/n!}\\
 & = \alpha~\vol(v_0).    
\end{align*}
Hence we prove the lemma.
\end{proof}

Following \cite{els03}, for a graded sequence of ideals $\fa_\bullet$
supported at $p$, the \textit{volume} of 
$\fa_\bullet$ is defined as
\[
 \vol(\fa_\bullet):=\limsup_{m\to\infty}\frac{\ell(R/\fa_m)}{m^n/n!},
\]
while the \textit{multiplicity} of $\fa_\bullet$ is defined as
\[
 \mult(\fa_\bullet):=\lim_{m\to\infty}\frac{\mult(\fa_m)}{m^n}.
\]
By \cite{lm09, cut13}, we know that $\mult(\fa_\bullet)=
\vol(\fa_\bullet)$ when $X$ is normal. 

\begin{thm}\label{compare2}
We have 
\begin{equation}\label{infs}
\inf_{\fa}\big(\lct(\fa)^n\cdot\mult(\fa)\big)=
\inf_{\fa_\bullet}\big(\lct(\fa_\bullet)^n\cdot\vol(\fa_\bullet)\big)
=\inf_{v}\hvol(v),
\end{equation}
where the infimums are taken over all ideals $\fa$ supported at $p$, 
all graded sequences of ideals $\fa_\bullet$ supported at $p$ and 
all real valuations $v\in\Val_{X,p}$, respectively. We also set 
$\lct(\fa_\bullet)^n\cdot\vol(\fa_\bullet)=+\infty$ if $\lct(\fa_\bullet)=+\infty$. 
\end{thm}

\begin{proof}
Let $\fa_\bullet$ be a graded sequence of ideal sheaves supported at
$p$. By \cite{jm12, bdffu}, we have
\[
\lim_{m\to\infty} m~\lct(\fa_m)=\lct(\fa_\bullet).
\]
Since $\mult(\fa_\bullet)=\vol(\fa_\bullet)$, if $\lct(\fa_\bullet)<+\infty$ then
\[
 \lim_{m\to\infty} \lct(\fa_m)^n\cdot\mult(\fa_m)=\lct(\fa_\bullet)^n
 \cdot\vol(\fa_\bullet).
\]
Thus we have
\begin{equation}\label{lct1}
 \inf_{\fa}\big(\lct(\fa)^n\cdot\mult(\fa)\big)\leq 
\inf_{\fa_\bullet}\big(\lct(\fa_\bullet)^n\cdot\vol(\fa_\bullet)\big).
\end{equation}

Next, Lemma \ref{compare1} implies that
\begin{equation}\label{lct2}
\inf_{\fa}\big(\lct(\fa)^n\cdot\mult(\fa)\big)\geq \inf_{v}\hvol(v). 
\end{equation}

Finally, for any real valuation $v\in\Val_{X,p}$ with $A_X(v)<+\infty$, by \cite[Theorem 1.2]{bdffu} we have
\[
\lct(\fa_\bullet(v))\leq\frac{A_X(v)}{v(\fa_\bullet(v))}=A_X(v)<+\infty.
\]
By definition we have $\vol(\fa_\bullet(v))=\vol(v)$. Therefore, 
\begin{equation}\label{lct2.5}
\lct(\fa_\bullet(v))^n\cdot\vol(\fa_\bullet(v))
\leq A_X(v)^n\cdot\vol(v)=\hvol(v).
\end{equation}
Hence we have
\begin{equation}\label{lct3}
\inf_{\fa_\bullet}\big(\lct(\fa_\bullet)^n\cdot\vol(\fa_\bullet)\big)\leq\inf_{v}\hvol(v).
\end{equation}
Combining \eqref{lct1}\eqref{lct2}\eqref{lct3} together, we finish
the proof.
\end{proof}

\begin{rem}
\begin{enumerate}[label=(\alph*)]
 \item The first equality in \eqref{infs} was observed by Musta\c t\u a
 in \cite{mus02} where he showed that 
 $\lct(\fa_\bullet)^n\cdot\mult(\fa_\bullet)\geq n^n$
 for any graded sequence of $\fm$-primary ideals $\fa_\bullet$ of a 
 regular local ring of dimension $n$.
 The inequality \eqref{lct2.5} was essentially realized by
 Li in \cite[Remark 2.8]{li15a} where he considered the smooth case.
 \item By Izumi's theorem, the infimums in Theorem \ref{compare2}
are always positive (see Section \ref{secval}).
 \end{enumerate}
\end{rem}

We will use Theorem \ref{compare2} to prove Theorem \ref{k-ss-dfem}
based on \cite{li15b,ll16}.

\begin{thm}[=Theorem \ref{k-ss-dfem}]\label{dfemtype}
 Let $V$ be a Fano manifold of dimension $n-1$.
 Assume $H=-rK_V$ is an ample Cartier divisor for 
 some $r\in\bQ_{>0}$. Let $X:=C(V,H)=\Spec\bigoplus_{k=0}^\infty
 H^0(V,mH)$ be the affine cone with cone vertex $o$.
 Then $V$ is K-semistable if and only if for any closed subscheme
 $Z$ of $X$ supported at $o$, the following inequality holds:
 \begin{equation}\label{eqdfemtype1}
  \lct(X;I_Z)^n\cdot\mult_Z X\geq 
  \frac{1}{r}((-K_V)^{n-1}).
 \end{equation}
 \end{thm}

\begin{proof}
 It is clear that $A_X(\ord_V)=r^{-1}$ and $\vol(\ord_V)=r^{n-1}((-K_V)^{n-1})$.
 Hence the right hand side of \eqref{eqdfemtype1} is equal to $\hvol(\ord_V)$.
 From Theorem \ref{compare2} we know that \eqref{eqdfemtype1}
 is equivalent to saying that the normalized volume $\hvol$ is minimized
 at $\ord_V$ over $(X,o)$.
 This is equivalent to $V$ being K-semistable by
 \cite[Corollary 1.5]{ll16}.
\end{proof}

\subsection{Finding minimizers}\label{minimizer}

Assume that the infimum of $\lct^n(\fa)\cdot\mult(\fa)$ is attained by 
some ideal $\fa=\fa_0$. Then Lemma \ref{compare1} and Theorem \ref{compare2}
together imply that the divisorial valuation $v_0$ that computes $\lct(\fa_0)$
is a minimizer of $\hvol$. In this subsection, we will study the converse problem,
i.e. assume there exists a minimizer $v_*$ of $\hvol$ that is 
divisorial, can the infimum of $\lct^n\cdot\mult$ for ideals be achieved?
In Proposition \ref{minlctmult}, we give an affirmative answer to this problem
assuming an extra condition on $v_*$.
\begin{defn}
Let $R$ be a Noetherian local domain. For a real valuation $v$ of $K(R)$ 
dominating $R$, we define the \textit{associate graded algebra} of $v$ as
\[
\mathrm{gr}_v R:=\bigoplus_{m\in\Phi}\fa_m(v)/\fa_{>m}(v),
\]
where $\Phi$ is the valuation semigroup of $v$.
\end{defn}

\begin{prop}\label{minlctmult}
Assume that $\hvol(\cdot)$ is minimized at a divisorial valuation $v_*$. If the graded algebra $\mathrm{gr}_{v_*}\cO_{X,p}$ is a finitely generated $\bC$-algebra, then for any sufficiently divisible $k\in\bZ_{>0}$ we have
\[
 \lct(\fa_k(v_*))^n\cdot\mult(\fa_k(v_*))=\hvol(v_*).
\]
In particular, the function $\big(\lct(\cdot)^n\cdot\mult(\cdot)\big)$ attains its minimum at $\fa_k(v_*)$.
\end{prop}

\begin{proof}
By Izumi's theorem we know that the filtration $\fa_1(v_*)\supset\fa_2(v_*)\supset\cdots$ defines the same topology on $\cO_{X,p}$ as the $\fm_p$-adic topology on $\cO_{X,p}$.
Then Lemma \ref{f.g.} implies $\oplus_{m\geq 0}\fa_m(v_*)$ is
a finitely generated $\cO_{X,p}$-algebra. Therefore, the 
degree $k$ Veronese subalgebra $\oplus_{m\geq 0}\fa_{km}(v_*)$ 
is generated in degree $1$ for $k\in\bZ_{>0}$ sufficiently divisible.
This means $\fa_{km}=\fa_k^m$ for any $m\geq 0$. From inequality \eqref{lct2.5} we have
\begin{align*}
 \hvol(v_*) & \geq \lct(\fa_\bullet(v_*))^n\cdot\vol(\fa_\bullet(v_*))\\
 & = \lct(\fa_{k\bullet}(v_*))^n\cdot\vol(\fa_{k\bullet}(v_*))\\
 & = \lct(\fa_k(v_*))^n\cdot\mult(\fa_k(v_*)).
\end{align*}
On the other hand, Theorem \ref{compare2} also implies that 
$\lct(\fa_k(v_*))^n\cdot\mult(\fa_k(v_*))\geq \hvol(v_*)$ because
$v_*$ is a minimizer of $\hvol(\cdot)$. 
So we prove the proposition.
\end{proof}

\begin{lem}\label{f.g.}
Let $(R,\fm)$ be a Noetherian local ring. Let $\fa_\bullet$ be a graded sequence of $\fm$-primary ideals. Assume that the linear topology on $R$ defined by the filtration $\fa_1\supset\fa_2\supset\cdots$ is the same as the $\fm$-adic topology on $R$. Assume that
$\oplus_{m\geq 0}\fa_m/\fa_{m+1}$ is a finitely generated $R/\fa_1$-algebra. Then $\oplus_{m\geq 0}\fa_m$ is a finitely generated $R$-algebra.
\end{lem}

\begin{proof}
We take a set of homogeneous generators of $\oplus_{m\geq 0}\fa_m/\fa_{m+1}$, say $x_1,\cdots,x_r$. Let $d_i$ be the degree of $x_i$, i.e. $x_i\in \fa_{d_i}/\fa_{d_i+1}$. Take $y_i\in\fa_{d_i}$ such that
$x_i=y_i+\fa_{d_i+1}$. Denote by $\oplus_{m\geq 0}\fb_m$ the graded $R$-subalgebra of $\oplus_{m\geq 0}\fa_m$ generated by $y_1,\cdots,y_r$. 

Since $\{x_i\}$ generates $\oplus_{m\geq 0}\fa_m/\fa_{m+1}$, we have $\fa_m=\fb_m+\fa_{m+1}$. Hence for any $l>0$ the induction yields
\[
\fa_m=(\fb_m+\fb_{m+1}+\cdots+\fb_{m+l-1})+\fa_{m+l}.
\]
Because $R$ is Noetherian, the ideals $(\fb_m+\fb_{m+1}+\cdots+\fb_{m+l-1})$ stabilize for $l$ sufficiently large. Denote the stabilized ideal by $\fc_m$, then we have $\fa_m=\fc_m+\fa_{m+l}$ for $l$ sufficiently large. Since the filtration $\fa_1\supset\fa_2\supset\cdots$ defines the same topology as the $\fm$-adic topology, we may take $l$ sufficiently large such that $\fa_{m+l}\subset\fm\fa_m$. As a result, we have $\fa_m=\fc_m+\fm\fa_m$ which implies that $\fa_m=\fc_m$ by Nakayama's lemma.

Now it suffices to show that $\oplus_{m\geq 0}\fc_m$ is a 
finitely generated graded $R$-algebra. By definition, 
$\fc_m$ is generated by monomials in $\{y_i\}$ of weighted 
degree at least $m$. For any $0\leq j\leq d_i$, denote by 
$y_{i,j}$ the homogeneous element $y_i\in\fa_{j}$ of degree
$j$. Then any monomial in $\{y_i\}$ of weighted degree at least
$m$ is equal to a monomial in $\{y_{i,j}\}$ of weighted degree $m$. Hence $\oplus_{m\geq 0}\fc_m$ is generated by $\{y_{i,j}\}$ as a graded $R$-algebra, which finishes the proof.
\end{proof}

\begin{ps}
 Since the first version of this article was posted on the arXiv,
 there has been a lot of progresses studying properties of the minimizer
 of $\hvol$. Here we mention two of them which are useful in our
 presentation.
 \begin{itemize}
  \item Blum \cite{blu16} proved that there always exists a minimizer
  of $\hvol$ for any klt singularity $x\in X$;
  \item If a divisorial valuation $v_*=\ord_S$ minimizes $\hvol$
  in $\Val_{X,x}$, then $S$
  is necessarily a \emph{Koll\'ar component} (see \cite{lx16} for a definition).
  In particular, $\mathrm{gr}_{v_*}\cO_{X,x}$ is finitely generated. This
  was proved by Blum \cite{blu16} and Li-Xu \cite{lx16} independently.
 \end{itemize}
\end{ps}

\section{Maximal volume cases}\label{maxvol}

\subsection{K-polystable varieties of maximal volume}

The following lemma is the main technical tool to prove the equality
case of Theorem \ref{quotsing} and Corollary \ref{ck}.

\begin{lem}\label{kpoly}
 Let $X$ be a K-polystable $\bQ$-Fano variety of dimension $n$.
 Let $p\in X$ be a closed point. Assume that there exists a 
 divisorial valuation $v_*\in\Val_{X,p}$ satisfying  $((-K_X)^n)=\left(1+\frac{1}{n}\right)^n
 \hvol(v_*)$.
 Then we have 
 \begin{enumerate}
 \item $v_*$ is a minimizer of the normalized volume function
 $\hvol$ on $\Val_{X,p}$;
 \item $X\cong \Proj~ (\mathrm{gr}_{v_*}\cO_{X,p})[x]$, where
 $x$ is a homogeneous element of degree $1$ (we may view $X$ as
 an orbifold projective cone over $F:=\Proj~\mathrm{gr}_{v_*}\cO_{X,p}$);
 \end{enumerate}
\end{lem}

\begin{proof}
 Part (1) is a direct consequence of Theorem \ref{mainthm2}.
 
 We need some preparation to prove part (2).
 Since $v_*$ is a divisorial minimizer of $\hvol$ by (1), it is induced by
 a Koll\'ar component by \cite[Proposition 4.9]{blu16} and \cite[Theorem 1.2]{lx16}. 
 In particular, $\mathrm{gr}_{v_*}\cO_{X,p}$ is a finitely generated
 $\bC$-algebra.
 Let $\bar{v}_*$ be the corresponding divisorial valuation
 on $X\times\bA^1$ (see \eqref{vbar}). Let $t$ be the parameter 
 of $\bA^1$. Let $\Spec R$ be an affine Zariski open subset of $X$.
 By Lemma \ref{f.g.},  $\oplus_{m\geq 0} \fa_m(v_*)$ is a finitely
 generated $R$-algebra.
 Then 
 \[
  \fa_m(\bar{v}_*)=\fa_m(v_*)+\fa_{m-1}(v_*)t+\cdots+\fa_{1}(v_*)
  t^{m-1}+(t^m).
 \]
 In order to keep track of the degree of the graded ring $\oplus_{m\geq 0}\fa_m(\bar{v}_*)$, 
 denote by $s$ the element $t$ in $\fa_1(\bar{v}_*)=\fa_1(v_*)+(t)$.
 Hence $\deg s=1$, $\deg t=0$ and we have
 \[
  \fa_m(\bar{v}_*)=\fa_m(v_*)+\fa_{m-1}(v_*)s+\cdots+\fa_{1}(v_*)
  s^{m-1}+s^m R[t],
 \]
 where $\deg \fa_{i}(v_*)=i$. As an $R[t]$-algebra, 
 $\oplus_{m\geq 0}\fa_m(\bar{v}_*)$ is generated by $\oplus_{m\geq 0}
 \fa_m(v_*)$ and $s$. Since $\oplus_{m\geq 0} \fa_m(v_*)$ is a finitely
 generated $R$-algebra, we have that 
 $\oplus_{m\geq 0}\fa_m(\bar{v}_*)$
 is of finite type over $\cO_{X\times \bA^1}$. 
 
 Let $\cX:=\Proj_{X\times\bA^1} \oplus_{m\geq 0}\fa_m(\bar{v}_*)$,
 hence $\cX$ is normal by Lemma \ref{primeF}. Denote by $g:\cX\to X\times\bA^1$ the projection. Denote
 the composite map by $\pi:\cX\to \bA^1$. The $\bQ$-line bundle $\cL$ 
 on $\cX$ is defined as
 \[
  \cL:=g^*(-K_{X\times\bA^1/\bA^1})+
  \left(1+\frac{1}{n}\right)A_X(v_*)\cO_{\cX}(1),
 \]
 where we treat $\cO_{\cX}(1)=\frac{1}{k}\cO_{\cX}(k)$ as a $\bQ$-line bundle
 for sufficiently divisible $k\in\bZ_{>0}$. 
 Let $\hX:=g_*^{-1}(X\times\{0\})$ and $\sigma=g|_{\hX}:\hX\to X$.
 Let $E$ be the exceptional divisor of $g$ given by $I_{E}:=
 \cO_{\cX}(1)$. Hence 
 \[
 E\cong \Proj ~\oplus_{m\geq 0}
 \fa_m(\bar{v}_*)/\fa_{m+1}(\bar{v}_*)=\Proj 
 ~(\mathrm{gr}_{v_*}\cO_{X,p})[s],
 \]
 where $s$ is a homogeneous element of degree $1$.
 Since $\mathrm{gr}_{v_*}\cO_{X,p}$ is an integral domain,
 $E$ is irreducible and reduced.
 
  We will show the following statements:
 \begin{enumerate}[label=(\alph*)]
  \item $\cL$ is $\pi$-nef with $\cL^\perp=\hX$;
  \item $\cL$ is $\pi$-semiample, hence $(\cX,\cL)$ is a semi $\bQ$-test configuration
  of $(X,-K_X)$;
  \item $\CM(\cX,\cL)=0$.
 \end{enumerate}

 For (a), we see that $\cL_t\cong -K_{\cX_t}$ is ample
 whenever $t\neq 0$. From the definition we know
 that $\cL$ is $g$-ample, hence $\cL|_E$ is ample. On $\hX$,
 we have that 
 \[
 \cL|_{\hX}\cong \sigma^*(-K_X)+ \left(1+\frac{1}{n}\right)
 A_X(v_*)\cO_{\hX}(1).
 \]
 For $k\in\bZ_{>0}$ sufficiently divisible, let $Z$ be the thickening of
 $p$ in $X$ such that $I_Z=\fa_k(v_*)$. From (1) we know that
 $v_*$ minimizes $\hvol$, hence $\lct(X;I_Z)^n\cdot\mult_Z X
 =\hvol(v_*)$ by Proposition \ref{minlctmult}. Since 
 \[
 ((-K_X)^n)=\left(1+\frac{1}{n}\right)^n
 \hvol(v_*)=\left(1+\frac{1}{n}\right)^n \lct(X;I_Z)^n\cdot\mult_Z X,
 \]
 by Lemma \ref{volume} and Remark \ref{volequality} we have that $\cL|_{\hX}$ is nef with
 zero top self intersection number. In addition, $\epsilon_Z(-K_X)
 =A_X(v_*)/m$. As a result, $\cL^\perp=\hX$.
 \medskip
 
 For (b), we first show that $\cX$ is normal $\bQ$-Gorenstein 
 with klt singularities. By Lemma \ref{primeF} we have that $E=
 \mathrm{Ex}(g)$ is a $\bQ$-Cartier prime divisor on $\cX$.
 Hence 
 \[
  K_{\cX}=g^*K_{X\times\bA^1}+(A_{\cX}(\ord_E)-1)E
 \]
 is $\bQ$-Cartier. To show that $\cX$ has klt singularities,
 it suffices to show that $(\cX,\hX)$ is plt.
 By Lemma \ref{primeF}, we have $\bar{v}_*=\ord_E$.
 Hence $\cX_0=\hX+E$ as Weil divisors since $\bar{v}_*(t)=1$.
 In particular, $K_{\cX}+\hX$ is $\bQ$-Cartier.
 
 It is clear that 
 \[
  I_E/(I_E\cdot I_{\hX})=I_E|_{\hX}=\cO_{\cX}(1)|_{\hX}=\cO_{\hX}(1)\subset\cO_{\hX}.
 \]
 Since the kernel of $I_E\to \cO_{\hX}$ is $I_E\cap I_{\hX}$,
 we have that $I_E\cdot I_{\hX}=I_E\cap I_{\hX}$.
 On the other hand, by computing graded ideals we know that
 $I_E+I_{\hX}=I_F$. Denote by $\eta$ the generic point of $F$,
 then $\cO_{\hX,\eta}=\cO_{\cX,\eta}/I_{\hX,\eta}$ is a DVR since $\hX$ is normal.
 Applying Lemma \ref{principal} to $(R,\fp,\fq)=(\cO_{\cX,\eta},
 I_{E,\eta}, I_{\hX,\eta})$ yields that $E$ is Cartier at $\eta$.
 Since $\cX_0=\hX+E$ is Cartier and $\hX\cap E=F$, we have that
 $\hX=\cX_0-E$ is Cartier in codimension $2$. Next we notice that 
 $\hX\to X$ produces a Koll\'ar component, hence $\hX$ has klt singularities.
 Thus  $(\cX,\hX)$ is plt by inverse of adjunction \cite[Theorem 5.50]{km98}.
 
 By Shokurov's base-point-free theorem, to show $\cL$ is $\pi$-semiample
 we only need to show that $\cL-K_{\cX/\bA^1}$ is $\pi$-ample.
 It follows from Lemma \ref{primeF} that $\bar{v}_*=\ord_E$ and $E\sim_{\bQ}\cO_{\cX}(-1)$.
 Hence we have 
 \begin{align*}
 -K_{\cX/\bA^1}& =g^*(-K_{X\times\bA^1/\bA^1})-
 (A_{X\times\bA^1}(\bar{v}_*)-1)E\\
 &= g^*(-K_{X\times\bA^1/\bA^1})+ A_X(v_*)\cO_{\cX}(1).
 \end{align*}
 Since $0<A_X(v_*)<(1+\frac{1}{n})A_X(v_*)$, we see that
 $-K_{\cX/\bA^1}$ is $\pi$-ample. Hence $\cL-K_{\cX/\bA^1}$ is $\pi$-ample.
 \medskip
 
 For (c), we know that 
 \[
  \CM(\cX,\cL)=\frac{1}{(n+1)((-K_X)^n)}\left(n
(\bar{\cL}^{n+1})+(n+1)(\bar{\cL}^n\cdot K_{\bar{\cX}/\bP^1})\right).
 \]
 By definition of $\cL$ we know that
 \begin{align*}
  \bar{\cL}& =\bar{g}^*(-K_{X\times\bP^1/\bP^1})+ \left(1+\frac{1}{n}
  \right)A_X(v_*)\cO_{\bar{\cX}}(1),\\
  K_{\bar{\cX}/\bP^1}& =\bar{g}^*K_{X\times\bP^1/\bP^1}-
  A_X(v_*)\cO_{\bar{\cX}}(1).
 \end{align*}
 Since $\cO_{\bar{\cX}}(1)$ is supported in $\bar{g}^{-1}((p,0))$,
 we have $(\bar{g}^*K_{X\times\bP^1/\bP^1}\cdot
 \cO_{\bar{\cX}}(1))=0$ as a cycle. Next, it is clear that
 $((\bar{g}^*K_{X\times\bP^1/\bP^1})^{n+1})=(K_{X\times\bP^1/\bP^1}^{n+1})
  =0.$
 Hence we have
 \begin{align*}
  (\bar{\cL}^{n+1})  & = \left(1+\frac{1}{n}\right)^{n+1} A_X(v_*)^{n+1}
  (\cO_{\bar{\cX}}(1)^{n+1}),\\
    (\bar{\cL}^n\cdot K_{\bar{\cX}/\bP^1})& = 
    -\left(1+\frac{1}{n}\right)^n A_X(v_*)^{n+1}
  (\cO_{\bar{\cX}}(1)^{n+1}).
 \end{align*}
Thus $\CM(\cX,\cL)=0$.
\medskip

Now we are ready to prove part (2).
By (b) we know that $\cL$ is semiample. Denote the ample model
of $(\cX,\cL)/\bA^1$ by $(\cY,\cM)$, with $h:\cX\to\cY$ and $\cL=h^*\cM$.
Then $(\cY,\cM)$ is a normal $\bQ$-test configuration of $(X,-K_X)$.
Since $\bar{h}_* K_{\bar{\cX}/\bP^1}=K_{\bar{\cY}/\bP^1}$,
we have $\CM(\cY,\cM)=\CM(\cX,\cL)=0$. 
From (a) we know that $\cL^\perp=\hX$, so $\mathrm{Ex}(h)=\hX$. Hence
$\cY_0=h_*\cX_0= h_*(\hX+E)=h_*E$.
In particular, $\cY_0$ is a prime divisor on $\cY$. It is clear
that $\cL|_E=(1+\frac{1}{n})A_X(v_*)\cO_E(1)$ is
ample, hence $h|_E: E\to \cY_0$ is a finite birational morphism.
Since $X$ is K-polystable by assumption, hence
$\CM(\cY,\cM)=0$ yields that $X\cong \cY_0$.
Then $h|_E: E\to \cY_0$ has to be an isomorphism
because $\cY_0\cong X$ is normal. Hence $X\cong E\cong\Proj~(\mathrm{gr}_{v_*}
\cO_{X,p})[s]$. 
\end{proof}

\begin{lem}\label{primeF}
 Let $(X,p)$ be a normal singularity. Let $v_*$ be a 
 divisorial valuation on $X$ centered at $p$. Assume that
 $\mathrm{gr}_{v_*}\cO_{X,p}$ is a finitely generated $\bC$-algebra.
 Let $F:=\Proj~ \mathrm{gr}_{v_*}\cO_{X,p}$. Define $\hX:=\Proj_X\oplus_{m\geq 0}
 \fa_m(v_*)$ with the projection $\sigma:\hX\to X$. Then
 $\hX$ is normal, $F=\mathrm{Ex}(\sigma)$ is 
 a $\bQ$-Cartier prime divisor on $\hX$, and $v_*=\ord_F$.
 (Note that $\sigma$ is called a prime blow-up in \cite{ish04}.)
\end{lem}

\begin{proof}
 By Lemma \ref{f.g.}, we know that $\oplus_{m\geq 0}\fa_m(v_*)$
 is a finitely generated $\cO_X$-algebra.
 Since the valuation ideals $\fa_m(v_*)$
 are always integrally closed, $\hX$ is normal.
 By definition of $F$ we know
 that the ideal sheaf $I_F$ is the same as the coherent sheaf
 $\cO_{\hX}(1)$. Let $k\in\bZ_{>0}$ be sufficiently divisible so that
 $\fa_{mk}(v_*)=\fa_k(v_*)^m$ for any positive integer $m$. Thus
 $\hX$ is naturally isomorphic to the blow up of $X$ along the ideal sheaf
 $\fa_{k}(v_*)$. In particular, $\cO_{\hX}(k)=\sigma^{-1}\fa_k(v_*)
 \cdot\cO_{\hX}$ is an invertible ideal sheaf. Next, $\cO_{\hX}
 /\cO_{\hX}(k)$ is supported at the exceptional locus of $\sigma$.
 As ideal sheaves on $\hX$,
 we know that $\cO_{\hX}(1)^k\subset\cO_{\hX}(k)\subset
 \cO_{\hX}(1)$. Hence $\cO_{\hX}(1)$ and $\cO_{\hX}(k)$ have the same
 nilradical as ideal sheaves,
 and 
 \[
  F_{\mathrm{red}}=\Supp(\cO_{\hX}/\cO_{\hX}(1))=\Supp(\cO_{\hX}/\cO_{\hX}(k))
 \]
 which is the reduced exceptional locus of $\sigma$.
 On the other hand, it is easy to see that $\mathrm{gr}_{v_*}\cO_{X,p}$
 is an integral domain, so $F$ is an integral scheme.
 
 For $k$ sufficiently divisible, we have that
 $\cO_{\hX}(k)$ is invertible and $\sigma_*\cO_{\hX}(mk)=\fa_{mk}(v_*)$
 for any $m\in\bZ_{\geq 0}$. Thus $\cO_{\hX}(k)=\cO_{\hX}(-lF)$ for some positive
 integer $l$ (so $F$ is $\bQ$-Cartier). Then for any $m\in\bZ_{\geq 0}$ we have the following
 equivalences:
 \begin{align*}
  v_*(f)\geq m & \Leftrightarrow v_*(f^k)\geq mk  \Leftrightarrow
  f^k \in \fa_{mk}(v_*) \Leftrightarrow \sigma^* f^k\in \cO_{\hX}(mk)\\
  & \Leftrightarrow \ord_F(\sigma^* f^k)\geq ml \Leftrightarrow
  \ord_F(f)\geq \frac{ml}{k}.
 \end{align*}
 Assume $v_*(f)=m$. If $\ord_F(f)>ml/k$, then we have $\ord_F(f)\geq
 (ml+1)/k$. Hence $\ord_F(f^l)\geq (ml+1)l/k$ which implies that
 $v_*(f^l)\geq ml+1$, a contradiction! Hence we have $\ord_F=(l/k)v_*$.
 Therefore, $v_*=\ord_F$ since both $v_*$ and $\ord_F$ are divisorial.
\end{proof}

\begin{lem}\label{principal}
  Let $(R,\fm)$ be a Noetherian local ring.
  Let $\fp,\fq$ be two prime ideals in $R$ satisfying
  that $\fp+\fq=\fm$, $\fp\cap\fq=\fp\fq$ and $R/\fq$ is a DVR.
  Then $\fp$ is principal.
\end{lem}
 
\begin{proof}
  Let $x+\fq$ be a uniformizer of $R/\fq$. Since $\fm=\fp+\fq$,
  we may choose $x$ so that $x\in \fp$.
  Hence we have $(x)+\fq=\fm$. As a result, 
  \[
   \fp=(x)+\fp\cap\fq=(x)+\fp\fq.
  \]
  So $\fp=(x)$ by Nakayama lemma.
\end{proof}

\subsection{Applications}
The following theorem improves Fujita's result on the equality case
in \cite[Theorem 5.1]{fuj15}.

\begin{thm}\label{Pn}
 Let $X$ be a Ding-semistable $\bQ$-Fano variety of dimension $n$.
 If $((-K_X)^n)\geq (n+1)^n$, then $X\cong\bP^n$.
\end{thm}

\begin{proof}[Proof 1]
 Notice that $((-K_X)^n)\leq (n+1)^n$
 by \cite[Corollary 1.3]{fuj15}. Thus we have $((-K_X)^n)=(n+1)^n$.
 Let $p\in X$ be a smooth point. From \cite[Proof of Theorem 5.1]{fuj15} or Remark
 \ref{volequality}, we see that $\epsilon_p(-K_X)=n+1$. Then
 the theorem is a consequence of the forthcoming paper
 \cite{lz16} (joint with Ziquan Zhuang), where we show that
 if an $n$-dimensional $\bQ$-Fano variety $X$ satisfies
 $\epsilon_p(-K_X)>n$ for some smooth point $p\in X$,
 then $X\cong\bP^n$.
\end{proof}

\begin{proof}[Proof 2]
 We follow the strategy and notation of the proof of Lemma \ref{kpoly}. Let $v_*
 :=\ord_p$ for a smooth point $p\in X$. As argued in the first proof, 
 the assumptions of Lemma \ref{kpoly} are fulfilled except that
 $X$ is only Ding-semistable rather than K-polystable. Nevertheless,
 we still have a semi $\bQ$-test configuration $(\cX,\cL)$ 
 of $(X,-K_X)$ such that $\cL^\perp=\hX$ and $\CM(\cX,\cL)=0$.
 Let $(\cY,\cM)$ be the ample model of $(\cX,\cL)/\bA^1$, with
 $h:\cX\to\cY$ and $\cL=h^*\cM$.
 We will show that $\cY_0\cong \bP^n$.
 
 Following the proof of Lemma \ref{kpoly}, 
 we know that $h|_E: E\to \cY_0$ is a finite birational morphism. Since 
 $v_*=\ord_p$ and $p\in X$ is a smooth point, we have that
 $E\cong \bP^n$.
 Therefore, $E$ is the normalization of $\cY_0$.
 Consider the short exact sequence 
 \[
  0\to\cO_{\cX}(-E)\to\cO_{\cX}\to\cO_E\to 0.
 \]
 By taking $h_*$, we get a long exact sequence
 \begin{equation}\label{exactpn}
  0\to h_*\cO_{\cX}(-E)\to \cO_{\cY} \to h_*\cO_E \to R^1 h_*
  \cO_{\cX}(-E)\to\cdots
 \end{equation}
 Since $\cX_0=\hX+E$, we have that $h_*\cO_{\cX}(-E)
 =\cO_{\cY}(-\cY_0)\otimes h_*\cO_{\cX}(\hX)$. Notice that
 $\hX$ is $h$-exceptional, hence $h_*\cO_{\cX}(\hX)=\cO_{\cY}$. As a result, $h_*\cO_{\cX}
 (-E)\cong \cO_{\cY}(-\cY_0)$. On the other hand, we have
 \[
 h^*\cM=\cL=g^*(-K_{X\times\bA^1/\bA^1})-(n+1)E=-K_{\cX/\bA^1}-E.
 \]
 So $-E=K_{\cX/\bA^1}+h^*\cM\sim_{\bQ, h} K_{\cX}$. 
 Since $\cX$ has klt singularities,  we have that 
 $R^1 h_*\cO_{\cX}(-E)=0$ by the generalized Kodaira
 vanishing theorem \cite[Theorem 10.19.4]{kol95}. Thus 
 the exact sequence \eqref{exactpn} yields that $h_*\cO_E=\cO_{\cY_0}$,
 which implies $\cY_0\cong E\cong \bP^n$.
  
Since $\bP^n$ is rigid under smooth deformation (cf. \cite[Exercise V.1.11.12.2]{kol96}),
we conclude that $\cY_t\cong\bP^n$ for general 
$t\in\bA^1\setminus\{0\}$, hence $X\cong\cX_t\cong\cY_t\cong\bP^n$.
\end{proof}

In the following theorem, we show the equality case of Theorem \ref{quotsing}.

\begin{thm}\label{quot2}
 Let $X$ be a K\"ahler-Einstein $\bQ$-Fano variety. Let $p\in X$
 be a closed point. Suppose $(X,p)$ is a quotient singularity
 with local analytic model $\bC^n/G$, where $G\subset GL(n,\bC)$ acts
 freely in codimension $1$. Then $((-K_X)^n)=(n+1)^n/|G|$ if and only if
 $|G\cap \mathbb{G}_m|=1$ and $X\cong\bP^n/G$.
\end{thm}

\begin{proof}
 For the ``if'' part, we may assume that $G\subset U(n)$. Hence
 $G$ preserves the Fubini-Study metric $\omega_{FS}$ on $\bP^n$.
 Since $|G\cap\mathbb{G}_m|=1$, the $G$-action on $\bP^n$ is free
 in codimension $1$. This implies that the quotient metric of 
 $\omega_{FS}$ on $\bP^n/G$ is K\"ahler-Einstein and $((-K_{\bP^n/G})^n)
 =((-K_{\bP^n})^n)/|G|=(n+1)^n/|G|$.
 \medskip
 
 For the ``only if'' part, let $(Y,o):=(\bC^n/G,0)$ be the local
 analytic model of $(X,p)$. Let $d:=|G\cap\mathbb{G}_m|$. Following the proof of Theorem \ref{quot1},
 we have a partial resolution $h:\hY\to Y$ of $Y$ that is the 
 quotient of the blow up $g:\widehat{\bC^n}\to\bC^n$. Since 
 $\hY$ is the quotient of $\widehat{\bC^n}$, it has klt singularities. Denote by $F$ the exceptional divisor of $h$, then we have $F\cong\bP^{n-1}/G$. Let $u_*:=\ord_F$ be the divisorial on $Y$ centered at $o$, then we have
 $\mathrm{gr}_{u_*}\cO_{Y,o}\cong\bC[x_1,\cdots,x_n]^G$ is a finitely
 generated $\bC$-algebra. (Here we assume $\deg x_i=1/d$ so that
 the grading is preserved under the isomorphism.) Theorem \ref{quot1} also implies
 $\hvol(u_*)=n^n/|G|$.
  
 Since $\widehat{\cO_{X,p}}\cong\widehat{\cO_{Y,o}}$, the divisorial valuation $u_*$
 induces a divisorial valuation $v_*$ on $X$ centered at $p$
 as in the proof  of Theorem \ref{quot1}. Thus $\hvol(v_*)=\hvol(u_*)$,
 $\mathrm{gr}_{v_*}(\cO_{X,p})
 \cong\mathrm{gr}_{u_*}(\cO_{Y,o})$ is a finitely generated $\bC$-algebra,
 $\hX:=\Proj_X\oplus_{m\geq 0}\fa_{m}(v_*)$ has klt singularities, and
 \[
  ((-K_X)^n)=\frac{(n+1)^n}{|G|}=\left(1+\frac{1}{n}\right)^n\hvol(v_*).
 \]
 By \cite{ber12} we have that $X$ is K-polystable. Hence applying Lemma \ref{kpoly}
 yields
 \[
     X\cong \Proj ~ \mathrm{gr}_{v_*}\cO_{X,p} [x] \cong \Proj 
     ~\bC[x_1,\cdots,x_n]^G[x],
 \]
 where $\deg x_i=1/d$ and $\deg x=1$. Denote $y:=x^{1/d}$, then 
 $\bC[x_1,\cdots,x_n]^G[x]$ is the Veronese subalgebra of 
 $\bC[x_1,\cdots,x_n]^G[y]$ 
 where $\deg x_i=\deg y=1/d$. As a result,
 \[
  X\cong \Proj~\bC[x_1,\cdots,x_n]^G[y] = \bP^n/G.
 \]

 Denote by $H_\infty$ the hyperplane at infinity in $\bP^n$. Let 
 $D$ be the prime divisor in $X$ corresponding to $H_\infty/G$ in $\bP^n/G$. It is clear
 that $(X, (1-\frac{1}{d})D)$ is an orbifold as a global quotient
 of $\bP^n$. Denote by $\pi:\bP^n\to X$ the quotient map, then we have
 \[
   \pi^*\left(K_X+\left(1-\frac{1}{d}\right)D\right)=K_{\bP^n},\quad \pi^*D=dH_\infty.
 \]
 Hence $\pi^* (-K_X)=-K_{\bP^n}+(d-1)H_\infty=\cO_{\bP^n}(n+d)$.
 This implies $((-K_X)^n)=(n+d)^n/|G|$, so $d=1$.
\end{proof}

\begin{rem}
\begin{enumerate}
 \item The restriction on $G$ in Theorem \ref{quot2} can be dropped in its
 logarithmic version as follows: suppose $(X,D)$ is a conical
 K\"ahler-Einstein log Fano pair, $p\not\in\Supp(D)$,
 $(X,p)$ is analytically isomorphic to $(\bC^n/G,0)$ and 
 $((-K_X-D)^n)=(n+1)^n/|G|$, then
 $(X,D)\cong(\bP^n/G,(1-\frac{1}{d})H_\infty/G)$
 where $d=|G\cap \mathbb{G}_m|$.
 \item The K-polystable condtion in Lemma \ref{kpoly} and K\"ahler-Einstein condition in
Theorem \ref{quotsing} cannot be dropped since any cubic surface in $\bP^3$ with only one or two $\bA_2$
singularities is K-semistable but not K-polystable (hence not K\"ahler-Einstein) according to \cite[Section 4.2]{oss16}, but all global quotients of $\bP^2$ are K\"ahler-Einstein (hence K-polystable).
\end{enumerate}
\end{rem}

\begin{cor}[=Corollary \ref{ck}]
Let $X$ be a K\"ahler-Einstein log Del Pezzo surface with at most Du Val
singularities. 
\begin{enumerate}
 \item If $((-K_X)^2)=1$, then $X$ has at most singularities of
 type $\bA_1$, $\bA_2$, $\bA_3$, $\bA_4$, $\bA_5$, $\bA_6$,
 $\bA_7$ or $\bD_4$.
 \item If $((-K_X)^2)=2$, then $X$ has at most singularities
of type $\bA_1$, $\bA_2$ or $\bA_3$.
 \item If $((-K_X)^2)=3$, then $X$ has at most singularities
of type $\bA_1$ or $\bA_2$.
 \item If $((-K_X)^2)=4$, then $X$ has at most singularities
of type $\bA_1$.
 \item If $((-K_X)^2)\geq 5$, then $X$ is smooth.
\end{enumerate}

\end{cor}

\begin{proof}
  Let $\bC^2/G$ be the local analytic model of $(X,p)$ for any 
 closed point $p\in X$. 
  By Corollary \ref{logdp}, we know that
 $|G|\leq 9/((-K_X)^2)$.  
  
  Recall that the order of orbifold group 
 at an $\bA_k$ singularity is $k+1$, a $\bD_k$ singularity
 is $4(k-2)$, an $\bE_6$ singularity is $24$, an $\bE_7$ 
 singularity is $48$, and an $\bE_8$ singularity is $120$.
 Hence the inequality $|G|\leq 9/((-K_X)^2)$ implies (2)(3)(4)(5).
 For part (1), the same argument yields that $X$ has at most
 singularities type of $\bA_k$ with $k\leq 8$, or of type $\bD_4$. Hence we only need to rule out
 $\bA_8$ cases. 
 
 Assume to the contrary that $(X,p)$ is of type
 $\bA_8$. Since $((-K_X)^2)=1=9/|G|$, Theorem \ref{quot2} 
 implies that $X\cong\bP^2/G$, where $G=\bZ/9\bZ$ acts on $\bP^2$
 as follows:
 \[
  [x,y,z]\mapsto [\zeta^i x,\zeta^{-i}y, z]\quad\textrm{for }
  i\in \bZ/9\bZ,
 \]
 where $\zeta:=e^{\frac{2\pi i}{9}}$ is the ninth root of unity.
 The point $p$ is the quotient of $[0,0,1]$ under this action,
 which is of type $\frac{1}{9}(1,-1)$ as a cyclic quotient singularity.
 However, the quotient singularities at $[1,0,0]$ and $[0,1,0]$ are both
 of type $\frac{1}{9}(1,2)$ which are not Du Val, contradiction!
\end{proof}


\begin{thebibliography}{99}

 \bibitem[Bat81]{bat81} V. V. Batyrev: {\it Toric Fano threefolds}. (Russian) Izv. Akad. Nauk SSSR Ser. Mat. 45 (1981), no. 4, 704-717, 927.

 \bibitem[Ber16]{ber12} Robert Berman: {\it K-polystability of $\bQ$-Fano 
 varieties admitting K\"ahler-Einstein metrics}.
 Invent. Math. 203 (2016), no. 3, 973-1025.
 
 \bibitem[BB11]{bb11} Robert Berman and Bo Berndtsson: {\it The projective
 space has maximal volume among all toric K\"ahler-Einstein manifolds}.
 Preprint available at \href{http://arxiv.org/abs/1112.4445}
 {\textsf{arXiv:1112.4445}}.

 \bibitem[BB17]{bb12} Robert Berman and Bo Berndtsson: {\it The volume of
 K\"ahler-Einstein Fano varieties and convex bodies}.
  J. Reine Angew. Math. 723 (2017), 127-152. 
  
 \bibitem[BBJ15]{bbj15} Robert Berman, S\'ebastien Boucksom and
 Mattias Jonsson: {\it A variational approach to the Yau-Tian-Donaldson
 conjecture}. Preprint available at \href{http://arxiv.org/abs/1509.04561}
 {\textsf{arXiv:1509.04561}}.
 
 \bibitem[BBEGZ11]{bbegz} Robert Berman, S\'ebastien Boucksom,
 Philippe Eyssidieux, Vincent Guedj and Ahmed Zeriahi:
 {\it K\"ahler-Einstein metrics and the K\"ahler-Ricci flow on log Fano
 varieties}. To appear in J. Reine Angew. Math.,
  available at \href{http://arxiv.org/abs/1111.7158}
 {\textsf{arXiv:1111.7158}}.
 
 \bibitem[Blu18]{blu16} Harold Blum: {\it Existence of valuations with smallest normalized volume}.
 Compos. Math. 154 (2018), no. 4, 820-849. 
 
 \bibitem[BC11]{bc11} S\'ebastien Boucksom and Huayi Chen: {\it
 Okounkov bodies of filtered linear series}. Compos. Math. 147 (2011),
 no. 4, 1205-1229.
 
 \bibitem[BdFFU15]{bdffu} S\'ebastien Boucksom, Tommaso de Fernex, Charles Favre and Stefano Urbinati: {\it Valuation spaces and multiplier ideals on singular varieties}.  Recent advances in algebraic geometry,  29-51, London Math. Soc. Lecture Note Ser., 417, Cambridge Univ. Press, Cambridge, 2015.

 \bibitem[BFJ14]{bfj14} S\'ebastien Boucksom, Charles Favre and
 Mattias Jonsson: {\it A refinement of Izumi's theorem}. Valuation
 theory in interaction, 55-81, EMS Ser. Congr. Rep., Eur. Math. Soc., Z¨¹rich, 2014.
 
 \bibitem[BHJ17]{bhj15} S\'ebastien Boucksom, Tomoyuki Hisamoto
 and Mattias Jonsson: {\it Uniform K-stability, Duistermaat-Heckman
 measures and singularities of pairs}.  Ann. Inst. Fourier (Grenoble) 67 (2017), no. 2, 743-841.
 
 \bibitem[BKMS14]{bkms14} S\'ebastien Boucksom, Alex K\"uronya, 
 Catriona Maclean and Tomasz Szemberg: {\it
Vanishing sequences and Okounkov bodies}. 
Math. Ann. 361 (2015), no. 3-4, 811-834. 

 \bibitem[Cam92]{cam92} F. Campana: {\it Connexit\'e rationnelle des vari\'et\'es de Fano}. (French) Ann. Sci. \'Ecole Norm. Sup. (4) 25 (1992), no. 5, 539-545. 

 \bibitem[Che08]{che08} Ivan Cheltsov: {\it Log canonical 
 thresholds of del Pezzo surfaces}. Geom. Funct. Anal. 18 (2008), no. 4, 1118-1144.
 
 \bibitem[CK14]{ck14} Ivan Cheltsov and Dimitra Kosta: {\it 
 Computing $\alpha$-invariants of singular del Pezzo surfaces}. 
J. Geom. Anal.  24  (2014),  no. 2, 798-842. 

 
 \bibitem[Cut13]{cut13} Steven Dale Cutkosky: {\it Multiplicities associated to graded families of ideals}. Algebra Number Theory 7 (2013), no. 9, 2059-2083.
 
 \bibitem[dFEM03]{dfem03} Tommaso de Fernex, Lawrence Ein and 
 Mircea Musta\c t\u a: {\it Bounds for log canonical thresholds with applications to birational rigidity}. Math. Res. Lett. 10 (2003), no. 2-3, 219-236.
 
 \bibitem[dFEM04]{dfem04} Tommaso de Fernex, Lawrence Ein and 
 Mircea Musta\c t\u a: {\it Multiplicities and log canonical threshold}. J. Algebraic Geom. 13 (2004), no. 3, 603-615.
 
 \bibitem[dFEM11]{dfem11} Tommaso de Fernex, Lawrence Ein and Mircea Musta\c t\u a: {\it Log canonical thresholds on varieties with bounded singularities}. Classification of algebraic varieties, 221-257, EMS Ser. Congr. Rep., Eur. Math. Soc., Z\"urich, 2011. 
 
 \bibitem[dFKL07]{dfkl07} Tommaso de Fernex, Alex K\"uronya and Robert Lazarsfeld:
 {\it Higher cohomology of divisors on a projective variety}. 
Math. Ann.  337  (2007),  no. 2, 443-455.

 \bibitem[dFM15]{dfm15} Tommaso de Fernex and Mircea Musta\c t\u a:
 {\it The volume of a set of arcs on a variety}. Rev. Roumaine Math. Pures Appl. 60 (2015), no. 3, 375-401. 
 \bibitem[Din88]{din88}Wei-Yue Ding: {\it Remarks on the existence problem of positive K\"ahler-Einstein metrics}, Math. Ann. 282 (1988), 463-471.

 \bibitem[DT92]{dt92} Wei Yue Ding and Gang Tian: {\it K\"ahler-Einstein
 metrics and the generalized Futaki invariant}. Invent. Math.  110  (1992),  no. 2, 315-335.
 
 \bibitem[Dol82]{dol82} Igor Dolgachev: {\it
 Weighted projective varieties}. Group actions and vector fields 
 (Vancouver, B.C., 1981), 34-71, 
 Lecture Notes in Math., 956, Springer, Berlin, 1982. 
 
 \bibitem[Don02]{don02} Simon Donaldson: {\it Scalar curvature and 
 stability of toric varieties}. J. Differential Geom. 62 (2002), 
 no. 2, 289-349.
 
 \bibitem[ELS03]{els03} Lawrence Ein, Robert Lazarsfeld and Karen E. Smith: {\it Uniform approximation of Abhyankar valuation ideals in smooth function fields}. Amer. J. Math. 125 (2003), no. 2, 409-440.
 
 \bibitem[Fuj15]{fuj15} Kento Fujita: {\it 
Optimal bounds for the volumes of K\"ahler-Einstein Fano manifolds}.
 To appear in Amer. J. Math., available at \href{http://arxiv.org/abs/1508.04578}
 {\textsf{arXiv:1508.04578}}.
 
 \bibitem[Fuj16]{fuj16} Kento Fujita: {\it 
 A valuative criterion for uniform K-stability of $\bQ$-Fano varieties}.
 To appear in J. Reine Angew. Math., available at \href{http://arxiv.org/abs/1602.00901}
 {\textsf{arXiv:1602.00901}}
 
 
 \bibitem[GK07]{gk07} Alessandro Ghigi and J\'anos Koll\'ar: {\it
 K\"ahler-Einstein metrics on orbifolds and Einstein metrics
 on spheres}. Comment. Math. Helv. 82 (2007), no. 4, 877-902.
 
 \bibitem[Ish04]{ish04} Shihoko Ishii: {\it Extremal functions and 
 prime blow-ups}. Comm. Algebra 32 (2004), no. 3, 819-827. 
 
 \bibitem[Izu85]{izu85} Shuzo Izumi: {\it A measure of integrity for
 local analytic algebras}. Publ. Res. Inst. Math. Sci. 21 (1985), 
 no. 4, 719-735.
 
 \bibitem[Jef97]{jef97} Thalia D. Jeffres: {\it
Singular set of some K\"ahler orbifolds}. 
Trans. Amer. Math. Soc.  349  (1997),  no. 5, 1961-1971. 

 
 \bibitem[JM12]{jm12} Mattias Jonsson and Mircea Musta\c t\u a: {\it Valuations and asymptotic invariants for sequences of ideals}. Ann. Inst. Fourier (Grenoble) 62 (2012), no. 6, 2145-2209 (2013).
 
 \bibitem[Kol95]{kol95} J\'anos Koll\'ar: {\it Shafarevich maps
 and automorphic forms}. M. B. Porter Lectures. Princeton University Press, Princeton, NJ, 1995. x+201 pp.
 
 \bibitem[Kol96]{kol96} J\'anos Koll\'ar: {\it Rational curves on 
 algebraic varieties}. Ergebnisse der Mathematik und ihrer Grenzgebiete. 3. Folge. A Series of Modern Surveys in Mathematics [Results in Mathematics and Related Areas. 3rd Series. A Series of Modern Surveys in Mathematics], 32. Springer-Verlag, Berlin, 1996. viii+320 pp.
  
 \bibitem[KMM92]{kmm92} J\'anos Koll\'ar, Yoichi Miyaoka, Yoichi
 and Shigefumi Mori: {\it Rational connectedness and boundedness of
 Fano manifolds}. J. Differential Geom. 36 (1992), no. 3, 765-779.  
  
 \bibitem[KM98]{km98}  J\'anos Koll\'ar and Shigefumi Mori:{ \it
 Birational geometry of algebraic varieties}. With the collaboration
 of C. H. Clemens and A. Corti. Translated from the 1998 Japanese original. Cambridge Tracts in Mathematics, 134. Cambridge University Press, Cambridge, 1998. viii+254 pp. 
  
 \bibitem[Laz04]{laz04} Robert Lazarsfeld: { \it Positivity in algebraic
 geometry. I. Classical setting: line bundles and linear series}.
 Ergebnisse der Mathematik und ihrer Grenzgebiete. 3. Folge. A
 Series of Modern Surveys in Mathematics [Results in Mathematics 
 and Related Areas. 3rd Series. A Series of Modern Surveys in
 Mathematics], 48. Springer-Verlag, Berlin, 2004. xviii+387 pp. 
  
 \bibitem[LM09]{lm09} Robert Lazarsfeld and Mircea Musta\c t\u a: {\it Convex bodies associated to linear series}. Ann. Sci. ¨¦c. Norm. Sup¨¦r. (4) 42 (2009), no. 5, 783-835.

 % \bibitem[Li13]{li13} Chi Li: {\it Yau-Tian-Donaldson correspondence for K-semistable Fano manifolds}. To appear in Crelle's Journal, available at
 %\href{http://arxiv.org/abs/1302.6681}{\textsf{arXiv:1302.6681}}.
 
  \bibitem[Li15]{li15a} Chi Li: {\it Minimizing normalized volumes of 
  valuations}. To appear in Math. Z., available at \href{http://arxiv.org/abs/1511.08164}
 {\textsf{arXiv:1511.08164}}.
 
 \bibitem[Li17]{li15b} Chi Li: {\it K-semistability is equivariant volume minimization}.
 Duke Math. J. 166 (2017), no. 16, 3147¨C3218.
 
 \bibitem[LL16]{ll16} Chi Li and Yuchen Liu: {\it K\"ahler-Einstein 
 metrics and volume minimizations}. To appear in Adv. Math.,
  available at 
 \href{http://arxiv.org/abs/1602.05094}
 {\textsf{arXiv:1602.05094}}.
 
 \bibitem[LX14]{lx14} Chi Li and Chenyang Xu: {\it Special test 
 configuration and K-stability of Fano varieties}. 
Ann. of Math. (2) 180 (2014), no. 1, 197-232. 

 \bibitem[LX16]{lx16} Chi Li and Chenyang Xu: {\it Stability of valuations and 
 Koll\'ar components}. Preprint available at 
 \href{http://arxiv.org/abs/1604.05398}
 {\textsf{arXiv:1604.05398}}.

 \bibitem[LZ16]{lz16} Yuchen Liu and Ziquan Zhuang: {\it Characterization
 of projective spaces by Seshadri constants}. To appear in Math. Z.,
 available at \href{http://arxiv.org/abs/1607.05743}{\textsf{arXiv:1607.05743}}.
 
 \bibitem[MM93]{mm93} Toshiki Mabuchi and Shigeru Mukai: {\it
Stability and Einstein-K\"ahler metric of a quartic del Pezzo
surface}. Einstein metrics and Yang-Mills connections (Sanda, 1990), 13360, 
Lecture Notes in Pure and Appl. Math., 145, Dekker, New York, 1993. 

 \bibitem[Mol1897]{mol1897} Theodor Molien: {\it \"Uber die Invarianten
 der linearen Substitutionsgruppe}. Sitzungsber. K\"onig.
Preuss. Akad. Wiss. (1897), 1152-1156. 
 
 \bibitem[Mus02]{mus02} Mircea Musta\c t\u a: {\it On multiplicities of graded sequences of ideals}. J. Algebra 256 (2002), no. 1, 229-249. 
 
 \bibitem[OSS16]{oss16} Yuji Odaka, Cristiano Spotti and Song Sun: {\it Compact moduli spaces of Del Pezzo surfaces and K\"ahler-Einstein metrics}. J. Differential Geom. 102 (2016), no. 1, 127-172.
 
 \bibitem[Ram73]{ram73} C. P. Ramanujam: {\it
 On a geometric interpretation of multiplicity}. 
Invent. Math.  22  (1973/74), 63-67.

 \bibitem[Ree87]{ree87} David Rees: {\it
Izumi's theorem}. Commutative algebra (Berkeley, CA, 1987), 407-416, 
Math. Sci. Res. Inst. Publ., 15, Springer, New York, 1989. 

 \bibitem[Shi10]{shi10} Yalong Shi: {\it On the 
 $\alpha$-invariants of cubic surfaces with Eckardt points}. 
 Adv. Math. 225 (2010), no. 3, 1285-1307. 

 \bibitem[Tia90]{tia90} Gang Tian: {\it On Calabi's conjecture for 
 complex surfaces with positive first Chern class}. Invent. Math. 101
 (1990), no. 1, 101-172.
 
 \bibitem[Tia97]{tia97} Gang Tian: {\it K\"ahler-Einstein metrics with
 positive scalar curvature}. Invent. Math. 130 (1997), no. 1, 1-37. 

% \bibitem[Rie74]{rie74} Oswald Riemenschneider: {\it Deformationen von Quotientensingularit\"aten (nach zyklischen Gruppen)}. (German) Math. Ann. 209 (1974), 211-248.
 
 \bibitem[WN12]{wn12} David Witt Nystr\"om: {\it Test configurations and Okounkov bodies}. Compos. Math. 148 (2012), no. 6, 1736-1756.
 
 \bibitem[Won13]{won13} Joonyeong Won: {\it Slope of a del Pezzo
 surface with du Val singularities}. Bull. Lond. Math. Soc.  45  (2013),  no. 2, 402-410. 
 
\end{thebibliography}
\end{document}